\documentclass[a4paper,reqno,10pt,oneside]{amsart}
\usepackage{amssymb,amsmath,amsthm,amsfonts}
\usepackage{geometry} 
\usepackage[T1]{fontenc}
\usepackage[english]{babel} 
\usepackage{fix-cm}
\usepackage{nicefrac} 
\usepackage[colorlinks]{hyperref}
\usepackage{graphicx}
\usepackage{braket}
\usepackage{graphicx}
\usepackage[usenames,dvipsnames]{xcolor}
\usepackage{cancel}
\usepackage{latexsym}
\usepackage{bm}
\usepackage[sc]{mathpazo}
\usepackage{ulem}

\newcommand{\R}{\mathbb{R}}

\newcommand{\eps}{\varepsilon}

\newcommand{\intr}{\int_{\R^{N}}}

\newcommand{\Ecal}{{\mathcal{E}}}

\newcommand{\Lcal}{{\mathcal{L}}}

\newcommand{\Ncal}{{\mathcal{N}}}

\newcommand{\Ocal}{{\mathcal{O}}}

\def\XXint#1#2#3{{\setbox0=\hbox{$#1{#2#3}{\int}$ }
\vcenter{\hbox{$#2#3$ }}\kern-.6\wd0}}

\newcommand{\wt}{\widetilde}

\newtheorem{proposition}{Proposition}[section]
\newtheorem{theorem}[proposition]{Theorem}

\newtheorem{lemma}[proposition]{Lemma}

\theoremstyle{definition}

\newtheorem{remark}[proposition]{Remark}

\numberwithin{equation}{section}
\newcommand{\beq}{\begin{equation}}
\newcommand{\eeq}{\end{equation}}
\newcommand{\ben}{\begin{enumerate}}
\newcommand{\een}{\end{enumerate}}
\newcommand{\bit}{\begin{itemize}}
\newcommand{\eit}{\end{itemize}}
\newcommand{\dys}{\displaystyle}

\newcommand{\uvmu}{\Upsilon}

\title{Partially concentrating standing waves for weakly coupled Schr\"odinger systems}
\begin{document}

\author{Benedetta Pellacci}
\address[B. Pellacci]{Dipartimento di Matematica e Fisica,
Universit\`a della Campania  ``Luigi Vanvitelli'',  via A.Lincoln 5, 81100
Caserta, Italy.}
\email{benedetta.pellacci@unicampania.it}
\author{Angela Pistoia}
\address[A. Pistoia]{Dipartimento SBAI, Sapienza Universit\`a di Roma,
via Antonio Scarpa 16, 00161 Roma, Italy. }
\email{angela.pistoia@uniroma1.it}
\author{Giusi Vaira}
\address[G. Vaira]{Dipartimento di Matematica, Universit\`a degli studi di Bari ``Aldo Moro'', via Edoardo Orabona 4, 70125 Bari, Italy.}
\email{giusi.vaira@uniba.it}
\author{Gianmaria Verzini}
\address[G. Verzini]{Dipartimento di Matematica, Politecnico di Milano, p.za Leonardo da Vinci 32,  20133 Milano, Italy.}
\email{gianmaria.verzini@polimi.it}
\thanks{Work partially supported by 
the MUR-PRIN-20227HX33Z ``Pattern formation in nonlinear phenomena'',
the project HORIZON EUROPE SEEDS – S51 STEPS: ``STEerability and controllability of PDES in Agricultural and Physical models",
the project "Start" within the program of the University "Luigi Vanvitelli" reserved to young researchers, Piano strategico 2021-2023, the MUR grant Dipartimento di Eccellenza 2023-2027, 
the INdAM-GNAMPA group.}



\begin{abstract}
We study the existence of standing waves for the following weakly coupled system of two Schr\"odinger equations
 \[
 \begin{cases}
 i \hslash \partial_{t}\psi_{1}=-\frac{\hslash^2}{2m_{1}}\Delta \psi_{1}+ {V_1}(x)\psi_{1}-\mu_{1}|\psi_{1}|^{2}\psi_{1}-\beta|\psi_{2}|^{2}\psi_{1}
 \medskip\\
 i \hslash \partial_{t}\psi_{2}=-\frac{\hslash^2}{2m_{2}}\Delta \psi_{2}+ {V_2}(x)\psi_{2}-\mu_{2}|\psi_{2}|^{2}\psi_{2}-\beta|\psi_{1}|^{2}\psi_{2},
 \end{cases}
 \]
where $V_1$ and $V_2$ are radial potentials bounded from below. We 
address the case $m_{1}\sim \hslash^2\to0$, $m_2$ constant, and 
prove the existence of a standing wave solution with both nontrivial components 
satisfying a prescribed asymptotic profile. In particular, the second component of 
such solution exhibits a concentrating behavior, while the first one keeps a 
quantum nature.
\end{abstract}

\maketitle
\section{Introduction}

The mathematical analysis of singularly perturbed semilinear elliptic equations and systems has been the object of a 
wide range of studies in the last decades. Among the many motivations, a big role is provided by models in Quantum 
Mechanics, and in particular by the semiclassical analysis of Schr\"odinger-type equations. In this context, one 
postulates that the classical Newtonian Mechanics should be recovered from the Quantum one by letting the Planck 
constant $\hslash$ vanish. Accordingly, the wave function of the quantum particle should concentrate and collapse 
to one or more Dirac's deltas, which position should describe the sharp location of classical particles. 
When different quantum waves interact, e.g. in the case of weakly coupled NLS systems, the 
commonly investigated  setting is the one in which all the waves concentrate in point particles. 
From the analytical point of view, this study  
lets different challenges arise. 
On the one hand, one may ask what is the limit 
concentrating profile at specific energy levels, for instance for ground states; this is typically done exploiting 
variational methods and blow-up analysis. On the other hand, solutions concentrating with prescribed shape and 
position can be constructed, mainly using the Lyapunov-Schmidt reduction approach. 

A largely studied model is the case of a binary mixture  of  Bose-Einstein condensates,
usually described by the Gross-Pitaevskii system,
namely a systems of  two weakly coupled  nonlinear Schrödinger equations
$$
\begin{cases}
 i \hslash \partial_{t}\psi_{1}=-\frac{\hslash^2}{2m_{1}}\Delta \psi_{1}+ {V_1}(x)\psi_{1}-\mu_{1}|\psi_{1}|^{2}\psi_{1}-\beta|\psi_{2}|^{2}\psi_{1} &\text{in $\R^{N}$}
 \medskip\\
 i \hslash \partial_{t}\psi_{2}=-\frac{\hslash^2}{2m_{2}}\Delta \psi_{2}+ {V_2}(x)\psi_{2}-\mu_{2}|\psi_{2}|^{2}\psi_{2}-\beta|\psi_{1}|^{2}\psi_{2} &\text{in $\R^{N}$},
\end{cases}
$$
 for $N=1, 2, 3$.  Here, 
$\psi_1$ and $\psi_2$ are the order parameters of the two components of the mixture, $m_1$ 
and $m_2$ the  corresponding masses, and $\widetilde V$ and $\widetilde W$ the external 
potentials, bounded from below. In general, the  (trapping) potentials may be different, opening 
interesting possibilities concerning the geometrical 
configurations of the condensates. Finally, the interaction parameters $\mu_1,\mu_2,\beta$ 
depend on  the scattering lengths associated to the different states. 
These interaction parameters can be fine-tuned across a wide rage of values, 
profiting from the presence of a Feshbach resonance. For more details on this model we refer to the book by Pitaevskii and Stringari \cite{pi-st}, in particular Chaps. 5 and 21. 

Looking for standing waves $(\psi_{1}(x,t),\psi_{2}(x,t))=
 (e^{iE_{1}t/\hslash}u(x), e^{iE_{2}t/\hslash}v(x))$
  of frequencies $E_{i}$ we have that $(u,v)$ solves 
  \begin{equation}\label{eq:sistsch1}
 \begin{cases}
 -\frac{\hslash^2}{2m_{1}}\Delta u+V(x)u=\mu_{1}u^{3}+\beta v^{2}u  &\hbox{in}\ \mathbb R^N,
 \medskip\\
-\frac{\hslash^2}{2m_{2}}\Delta v+W(x)v=\mu_{2}v^{3}+\beta u^{2}v  &\hbox{in}\ \mathbb R^N,
 \end{cases}
 \end{equation}
 with $V(x)=E_{1}+\widetilde{V}(x)$ and  $W(x)=E_{2}+\widetilde{W}(x)$,
  for $E_{i}$ such that $\inf_{\R^{N}}V>0$  and  $\inf_{\R^{N}}W>0$.
The usually studied case is the one in which $\mu_{i},$ $\beta$ are both very large, or $\hslash$ is very small with respect to the other parameters, so that one is lead to consider
the singularly perturbed elliptic system
 \begin{equation}\label{non-au} 
\begin{cases}
-\varepsilon^2 \Delta u +V_{1} (x)u=\mu_1 u^3+\beta uv^2 &\hbox{in}\ \mathbb R^N,
\\
-\varepsilon^2 \Delta v+V_{2}(x)v=\mu_2 v^3+\beta u^2v & \hbox{in}\ \mathbb R^N.
\end{cases} 
\end{equation}
While the autonomous case, namely the case $V_{i}\equiv \lambda_{i}$ with $\lambda_{i}$ positive constants,   has been widely studied in the last two decades, (see
 the recent paper  \cite{wei-wu} for an exhaustive list of references), a few  results concerning the non-autonomous situation are known.\\
  The first result seems due to  Lin and Wei \cite{lin-wei}
  who studied the case of a binary mixture in a   singularly perturbed regime and in presence of trapping potentials.  
 They prove the existence of   ground state solutions, derive their asymptotic behaviors as $\varepsilon\to0$ and show that each component   has one maximum point (possibly the same), called {\em spike}, which is trapped at the minimum points of the potentials $V_{i}$.
 The existence of a concentrating ground state solutions  was also established  by Montefusco,   Pellacci and Squassina \cite{MPS}, Pomponio \cite{pom},
Ikoma and Tanaka \cite{iko-tan} and 
Byeon \cite{byeon}.
All the previous papers are concerned with system \eqref{non-au} where both the equations are affected by the presence of  the small parameter $\varepsilon$, so that every wave (given by the components of the vector solution)  concentrates  as $\varepsilon$ approaches zero. 

Here, we focus on another type of regime, as
it may happen that only some of the waves act in a semiclassical way, while the others persist in a quantum behavior. 
To the best of our knowledge, this kind of analysis is not present in the PDEs literature yet, and this paper is a first contribution in this direction. 

More precisely, in our study we consider $2m_{1}= \hslash^2$ and $m_{2}=\frac12$ in \eqref{eq:sistsch1}, so that 
 $(u,v) $ 
solves  the following elliptic weakly coupled system
\begin{equation}\label{pro:P0}
\begin{cases}
-\Delta u+V(x)u=\mu_{1}u^{3}+\beta uv^{2} &\text{ in } \R^{N},
\\
-\eps^{2}\Delta v+W(x)v=\mu_{2}v^{3}+\beta vu^{2}&\text{ in } \R^{N}, 
\end{cases}
\end{equation}
where $\eps^2:=\hslash^2$. According to the previous discussion, along this paper we deal 
with the system above  in the singularly perturbed regime $\eps\to0$.
Moreover, our study will deal with the case of   $\mu_{i}>0$ and $\beta<0$, corresponding to positive intraspecies and  to a negative interspecies scattering length, describing a repulsive interaction between the condensates.

Our analysis will include the class of potentials satisfying the following assumptions.
\begin{itemize}
\item[$\bf{(V_{1})}$] $V(x)=V(|x|)$ is a radially symmetric function satisfying
\begin{equation} \label{eq:V0}
V \in C^{0}(\mathbb R^N) \ \hbox{and}\  \inf_{\R^{N}}V(x)>0.
\end{equation}
Moreover, $V$ is  such that  there  exists $ \uvmu \in C^{3}(\R^{N})\cap H^{2}(\R^{N})$  unique positive  radial solution to
\begin{equation} \label{eq:V1}
\begin{cases}
-\Delta \uvmu+V(x)\uvmu=\mu_1 \uvmu^{3}
\\
\uvmu(x)\to 0 \text{ as } |x|\to +\infty,
\end{cases}
\end{equation}
which is non-degenerate in the space $H^{1}_{e}(\R^{N})$ of functions even with respect to each variables, i.e.    
\begin{equation}\label{he}
H^{1}_{e}(\R^{N})\!:=\!\left\{u\in H^{1}(\R^N) : u (\!x_{1},\dots, x_{i},\dots x_{N}\!)\!=\!u (\!x_{1}, \dots,-x_{i}, \dots, x_{N}\!), \,i=1,...\,,N\!
\right\}
\end{equation}
that is  the only solutions to
$$\left\{\begin{aligned}
&- \Delta z +V(x)z =3\mu_1 \uvmu^2 z\ \hbox{in}\ \mathbb R^N\\
&z\in H^{1}_{e}(\R^{N})\\
\end{aligned}\right.$$
is the trivial one.\\

\item[$\bf{(V_{2})}$] 
The potential $V$ is such that for every 
$f\in L^m(\R^N)$, $ 2\le m <+\infty$, there exists a unique solution 
$u\in W^{2,m}(\R^N)$ of the equation
\[
-\Delta u+V(x)u=f,
\]
and
 \[
\|u\|_{W^{2,m}(\R^{N})}\lesssim \|f\|_{L^{m}(\R^{N})}.
\]

\item[$\bf{(W)}$] $W(x)$ is even with respect to all the variables, 
\begin{equation} \label{eq:W0}
W \in C^{3}(\R^{N})\ \hbox{and}\ \inf_{\R^{N}}  W(x)> 0.
\end{equation}
\end{itemize}

 Our main result is stated as follows.
\begin{theorem}\label{thm:principal}
Let $N=2,3$, and suppose that ${\bf (V_{1}), (V_{2})}$ and ${\bf (W)}$ hold. \\
Assume $\beta<0$. Let 
\begin{equation}\label{new-pot}
\omega(x):=W(x)-\beta^2 \uvmu (x), 
\end{equation}
set  $\omega_0:=\omega(0)>0$ and assume that 
\begin{equation}\label{min-pot} 
\dfrac{\partial^2\omega}{ \partial x_1^2} (0)<0.\end{equation} 
Then there exists $\eps_{0}>0$ such that for every $\eps\in (0,\eps_{0})$ there exists a solution
$(u_{\eps},v_{\eps})$ of system \eqref{pro:P0}
   even  with respect to each variable and  having the following asymptotic profile as $\eps\to0$
\begin{equation}\label{eq:profile}
u_{\eps}(x)\sim  \uvmu(x) ,\qquad
 v_{\eps}(x)\sim \sqrt{\frac{\omega_{0}  }{\mu_{2}}}
\left[U\left(\omega_{0}  \frac{x-P_{\eps}}\eps\right)+U\left(\omega_{0}  \frac{x+ P_{\eps}}\eps\right)\right],
\end{equation}
where   $\uvmu$ solves \eqref{eq:V1}, $U$   is the positive radial solution of
\begin{equation} 
-\Delta U+\omega_0  U=\mu_2 U^3\ \hbox{in}\ \R^N
\end{equation}
and  the peaks $ P_{\eps}$ and $-P_\eps$  collapse to the origin as
\begin{equation}\label{eq:peps}
\pm P_{\eps}=\rho_{\eps}  (\pm 1,0,\dots,0), \; \text{ with }\;\lim_{\eps\to0^{+}}\frac{\rho_{\eps}}{\eps\ln(1/\eps)}= \frac{1}{\sqrt{\omega_0  }}.
\end{equation}
\end{theorem}

Theorem \ref{thm:principal} states the existence of a solution whose first component looks like 
a genuine solution to \eqref{eq:V1}, in particular it does not concentrate, and whose second component concentrates at two  opposite points which collapse to the origin  as $\eps$ goes to 0.
As a consequence of the coupling in the equations, the first component of \eqref{pro:P0}  plays the  role of  an additional potential  in the singularly perturbed second equation, 
so that the concentration   will be triggered by the {\em modified  potential} $W-\beta \Upsilon^{2}$.\\
Because of the assumption  $\beta<0$, we will obtain a solution in the repulsive regime, and,  when the origin 
is a maximum point of $W$ our solution exists for every $\beta$ negative; 
while when $\partial^{2}_{11}W(0)>0$ we obtain a solution for $\beta< \beta_{0}<0$ (see \eqref{eq:ipobeta}).
\\

Let us make some comments.
\\

\begin{remark}\label{rem:deltaY(0)not0}
We point out that, in case $\uvmu$ satisfies  
\begin{equation}\label{ipotesi}
\Delta\uvmu(0) \not=0,
\end{equation}   
then assumption  \eqref{min-pot} is satisfied as long as
 \begin{equation}\label{eq:ipobeta}
\beta<  
-\frac{{\partial^{2}W\over {\partial x_1}}(0)}{2\uvmu(0)|{\partial^{2}\uvmu\over {\partial x_1}}(0)|}.
\end{equation}
On the other hand, assumption \eqref{ipotesi} is verified in case $V$ is radially non-decresing near $0$ and $\uvmu$ is radial and has a (local) strict 
maximum at $0$, as one can verify applying Hopf's Lemma. Indeed, 
first we observe that, in such a case, for any $i$ the function $ \partial_i \uvmu$ solves
$$\Delta \partial_i \uvmu-V\partial_i \uvmu=-3\uvmu^2\partial_i \uvmu+\uvmu\partial_i V\ge0\ \hbox{in}\ \Omega_i:=\left\{|x|\le r_{0}\ :\ x_i\ge0\right\},$$
for $r_0$ small, because $\partial_i \uvmu\le 0$ and  $ \partial_i V\ge0$ in $\Omega_i$.
Moreover
$$0=\partial_i \uvmu(0)=\max\limits_{\Omega_i}\partial_i \uvmu.$$
By Hopf's Lemma we deduce $\partial_{ii}\uvmu(0)<0 $ and \eqref{ipotesi} follows.
\end{remark}

\begin{remark}\label{rema1}
It is useful to recall the classical results concerning the case of  
 constant potential, i.e.
 \begin{equation}\label{Uc}
-\Delta U+\lambda U=\mu U^{3}\ \hbox{in}\ \R^N,\ U\in H^1(\R^N).
\end{equation}
It is well known that \eqref{Uc} has  an unique positive solution which is radially symmetric   and also that
the  set of solution of the corresponding linearized 
equation
$$-\Delta z+\lambda z=3\mu U^{2}z\ \hbox{in}\ \R^N,\ z\in H^1(\R^N)$$
 is spanned by the $N$ partial derivatives  $ {\partial U\over \partial x_i}$ which are  odd in each variable.
In addition,   $U$    is radially decreasing and it satisfies the following exponential decay (see \cite{beli1,beli2, kwong})
\begin{equation}\label{eq:Udecay}
\lim_{|x|\to\infty}U(x)e^{\sqrt{\lambda}|x|}|x|^{\frac{N-1}2}=C_{0}>0,
\qquad
\lim_{|x|\to\infty}\frac{U'(x)}{U(x)}=-1.
\end{equation}
 \end{remark}
 
\begin{remark}
We observe that a class of potentials $V$  which satisfy  hypotheses ${\bf (V_{1})}$ and ${\bf (V_{2})}$ includes both the constant potentials and, 
at least in dimension $N=3$, the trapping ones, like  
$V(x)=\lambda+ |x|^{m}$ for some $\lambda>0$ and $m>0$. 
Indeed, ${\bf (V_{1})}$ follows from Remarks \ref{rema1}, \ref{rem:deltaY(0)not0} and \cite[Corollary 1.5 and Appendix A]{bo}.
On the other hand, ${\bf (V_{2})}$ follows by standard elliptic theory in the case of constant potentials (see e.g. 
\cite[Thm. 3, p. 135]{MR290095}), while for trapping ones it is a consequence of \cite{MR1004744,MR970154,shen} (as we mentioned, \cite{shen} deals only with dimensions $N\ge3$; we believe that a version of such results should hold also in lower dimension, but this is far beyond the aim of this paper).
\end{remark}

  \begin{remark}
Our result relies on the simmetry of the potentials $V$ and $W$ which allows to build symmetric solutions  with  symmetric peaks
 $P_\varepsilon$ and $-P_\varepsilon$ collapsing to the origin.
We strongly believe that a   similar construction could be carried out in a more general setting in the spirit of Kang  and Wei  \cite{kw},
 when the radial solution of the first equation \eqref{eq:V1} is non-degenerate in the whole space  $H^1(\mathbb R^N)$ (i.e. $V$ is a trapping potential as in \cite{bo, shen}).  In that case, it should be possible to build a solutions whose first component resembles the radial solution $\uvmu$ of \eqref{eq:V1} and the second component has two different peaks collapsing to a maximum point of the modified potential 
 $W -\beta \uvmu ^2 $.
 \end{remark}

\begin{remark}
We will prove Theorem \ref{thm:principal}  using a classical Lyapunov-Schmidt reduction.
This will allow us to build each component with a prescribed profile: 
the first component will look as one bump solution for $\varepsilon $  sufficiently small,
while the second will  develop   two spikes  collapsing at the origin and will be exponentially 
small far from them.
In performing this classical procedure, we faced some new difficulties. First of all,
in view of the square growth of the coupling term,  we need to correct the ansatzs of both the components to detect the suitably reduced problem.
Moreover, the use of the regularity theory will be crucial  in order to make suitable expansion of all the  terms involved in the construction.
\\
Our existence result does not cover the case $N=1$, as  in this case the size of the error term does not produce the suitable smallness of the reminders terms in the ansatz despite of the presence of the correction term. We think that this  point could be managed introducing further correction terms, again in both the equations, which at the prices of heavy technicality should allow to construct a remaining term sufficiently small. 
\end{remark}

\begin{remark}
Our result deals with the case of a binary mixture and it is natural to ask  
 if our construction can be extended to the case of a larger number of equations, i.e.
$$ 
\begin{cases}
-  \Delta u +V (x)u=\mu_1 u ^3+ u \sum\limits_{j=1}^k \beta_{j} v_j^2 &
 \hbox{in}\ \mathbb R^N,
 \\
-\eps^2 \Delta v_j +W_j(x)v_j=\mu_j v_j^3+v_j  \left(\beta_j u^2+ \sum\limits_{i\not=j}^k \beta_{ij} v_i^2\right) & \hbox{in}\ \mathbb R^N,\ j=2,\dots,k.
\end{cases}
$$
In particular, we wonder if it is possible to build a solution whose components $v_j$ concentrate at different pairs of points $(P_j^\eps,-P_j^\eps)$ collapsing to the origin as $\eps$ goes to zero.

\end{remark}

\begin{remark}

The unperturbed version of  system  \eqref{non-au} (let us say $\eps=1$) was firstly studied by Peng and Wang \cite{peng-wang}  who (in presence of radial potentials)
 constructed an unbounded sequence of non-radial 
solutions exhibiting an arbitrarily large number of peaks. Their result has been successively extended to the case of more than two equations by  
Pistoia and Vaira \cite{pv} and  very recently by Li, Wei and Wu \cite{li-wei-wu}.\\
We wonder if it is possible, by combining the ideas used in the above papers, to produce a solution to system  \eqref{pro:P0} with $\eps=1$ whose first component looks like the solution to  \eqref{eq:V1} and second component concentrates at an arbitrary large number of points approaching infinity as $\eps$ goes to zero.

\end{remark}

\begin{remark}
Let us finally observe that the existence of solutions to the system \eqref{non-au} is closely related to the study of  the normalized solutions  
for nonlinear Schrödinger systems. We refer the reader to the recent papers  by Lu \cite{lulu}, Liu and Yang \cite{liu-yang}, Guo and Xie \cite{guo-xie} and Liu and Tian \cite{liu-tian}.  
In particular, it would be interesting to produce normalized solutions using as a parameter their $L^2$-norms, in the spirit of the results
obtained by Pellacci, Pistoia, Vaira and Verzini \cite{ppvv}.

\end{remark}

The paper is organized as follows.
In the next section we  set  the problem, 
by introducing the  the main 
blocks of our construction and by reformulating problem \eqref{pro:P0}  as a system of two equations, one set in an infinite dimensional set, 
the other, called the reduced problem, in a  finite dimensional one.
In Section \ref{sec:equainfdim} we solve the  the infinite dimensional equation.
Finally, in Section \ref{reduced} we study the reduced problem and we complete the proof of Theorem \ref{thm:principal}.
\\

\textbf{Acknowledgments.}
The authors warmly thank the anonymous referee for her/his precious comments, and in particular for having pointed out a gap in the proof of Lemma \ref{phi}  in a previous version of this manuscript.


\section{Setting of the problem}\label{setting}

Let us introduce the Banach spaces
\begin{equation}\label{eq:defXV}
\begin{split}
H^2_{V}
&
:=\left\{u\in H^2(\R^{N}) :  \intr V(x) u^{2}<+\infty\right\}, 
\\
H^2_{W_{\eps}}&
:=\left\{u\in H^2(\R^{N}) :  \intr W(\eps x) u^{2}<+\infty\right\} .
\end{split}
\end{equation}
equipped with the norms
$$\|u\| _V:=\left(\intr\sum\limits_{|\alpha|=2}|D^\alpha u|^2+\intr  |\nabla u|^2+\intr V(x) u^{2}\right)^\frac12$$
and
$$
\|u\| _\eps:=\left(\intr\sum\limits_{|\alpha|=2}|D^\alpha u|^2+\intr  |\nabla u|^2+\intr W(\eps) (x) u^{2}\right)^\frac12
$$
Henceforth, we omit the subscript $\eps$ in $u,\,v$ and  we agree that $a\lesssim b$ means $|a|\le c |b|$ for some constant $c$ which does not depend on $a$ and $b$.
\\
Performing a change of variable in the second equation,  we are lead  to seek a solution $(u,v) $ of  
\begin{equation}\label{pro:P2}
\begin{cases}
-\Delta u+V(x)u=\mu_{1}u^{3}+\beta u{v}^{2}\left(\frac{x}\eps\right) &\text{ in } \R^{N}, 
\\
-\Delta v+W(\eps x)v=\mu_{2}v^{3}+\beta u^{2}(\eps x)v &\text{ in } \R^{N}.
\end{cases}
\end{equation}  in  the space 
\[
X=\{(u,v)\in H^2_V\times H^2_{W_\eps}\ :\ \hbox{$u,v$ are  even functions (see \eqref{he})}\}.
\]
In the next subsection,  we introduce the main building blocks in the construction of our solution.
\subsection{The ansatz and the correction terms.}
 In view of assumptions $\bf{(V_{1})}$ and $\bf{(V_{2})}$, we can consider   $\uvmu$  the 
 solution of
\begin{equation}\label{uvmu}
-\Delta \Upsilon+V(x) \Upsilon=\mu_1 \Upsilon^3\ \hbox{in}\ \R^N
\end{equation}
and      $U$   be the solution of
\begin{equation}\label{eq:Ulambdamu}
-\Delta U+\omega_0  U=\mu_2 U^3\ \hbox{in}\ \R^N,
\end{equation}
where, since $\beta<0$
\[
\omega_0  :=W(0)-\beta \uvmu^{2}(0)>0.
\]
We look for a solution $(u, v)$ of \eqref{pro:P2} of the form 
\
\begin{equation}\label{ans}
u(x)=\underbrace{\uvmu(x)+\beta\Phi_\eps (x)}_{=:\Xi_\eps (x)} +\varphi(x),\qquad
v(x)=\underbrace{ U_\eps (x)+\beta\Psi_\eps(x)}_{:=\Theta_\eps(x)}+\psi(  x).
\end{equation}
where
$$ U_\eps(x):= U  \left(x-\frac{P_{\eps}}{\eps} \right)
+U  \left(x+ \frac{P_{\eps}}{\eps}\right)=U_{-P_{\eps}}(x)+U_{P_{\eps}}(x)$$
and
the  concentration points satisfy (see \eqref{eq:peps})
\begin{equation}\label{punti}
P_{\eps}=\rho_{\eps}P_{0}=\rho_{\eps} (1,0,\dots,0), \ \rho_{\eps} =d{\eps\ln(1/\eps)},\ d\in\left(\frac1{\sqrt{\omega_{0}}}-\delta,\frac1{\sqrt{\omega_{0}}}+\delta\right)\ \hbox{for some}\ \delta>0.
\end{equation}
The function $\Phi_\eps $ and $\Psi_{\eps}$, are suitable correction terms, 
whose existence and properties are established in Lemma \ref{phi}, \ref{le:Psi}. 
The reimander terms $\varphi$ and $\psi$ belong to the space 
\[
  K^{\perp} :=\left\{(\varphi,\psi)\in X\ :\ \intr \psi(x)Z_\eps(x)dx=0\right\},
  \]
where \(  K:=\mathtt{span}\{(0,Z_\eps)\} \subset H^{2}\)
and 
\begin{equation}\label{eq:defU2}
Z_\eps (x):=
\partial_1 U  \Big(x+\frac{P_{\eps}}{\eps}\Big)-\partial_1 U  \Big(x-\frac{P_{\eps}}\eps\Big),
\end{equation}
 solves the linear equation
 \begin{equation}\label{eq:eqZ}
 -\Delta Z_\eps   +\omega_0  Z_\eps  =3\mu_2\left(U^2_{P_\eps}\partial_{x_1}U_{P_\eps}- U^2_{-P_\eps}\partial_{x_1} U_{-P_\eps}\right).
 \end{equation}
Moreover, it is worthwhile to point out that all the functions $\Phi_\eps ,$ $U_\eps$ and $Z_\eps $ are even functions.

In the following we introduce the two correction terms we need in our construction of the solution.
Let us start from the term in the first component.
\begin{lemma}\label{phi} 
There exists a unique  even  $\Phi_\eps\in H^2_V(\R^N)$ solution of the equation
\begin{equation} \label{eq:V}
-\Delta \Phi_\eps  +\left(V(x)-3\mu_{1}\uvmu^{2}(x)\right)\Phi_\eps 
= \uvmu(x) U_\eps  ^{2}\left(\frac{x}\eps\right).
\end{equation}
Moreover, $\Phi_{\eps}$ satisfies  $\|\Phi_\eps\|_{W^{2,m}(\R^N)}\lesssim \eps^{N\over m}$ and 
$ \|\Phi_\eps\|_{C^{1,1-\frac N m}(\R^N)}\lesssim \eps^{N\over m}$ for any $m\ge2.$
\end{lemma}
\begin{proof}
By exploiting assumptions  $\bf{(V_{1})}-\bf{(V_{2})}$ we deduce that  for any even function $f\in L^2(\R^N)$ the problem
\begin{equation*}
-\Delta \Phi   +\left(V(x)-3\mu_{1}\uvmu^{2}(x)\right)\Phi  = f\ \hbox{in}\ \R^N
\end{equation*}
has a unique even solution $\Phi$ such that
 $\|\Phi\|_V\le c \|f\|_{L^2(\R^N)},$ for some constant $c$ which does not depend on $f$.
Now we point that the function $f(x)=\uvmu(x) U_\eps  ^{2}\left(\frac{x}\eps\right)$ is an even function with
$$\left\|\uvmu U_\eps  ^{2}\left(\frac{\cdot}\eps\right)\right\|_{L^2(\R^N)}\lesssim \eps^{\frac N2}.$$
Indeed, by scaling $x=\eps y$ we immediately deduce
\begin{equation}\label{eq:ueps}
\intr \uvmu^2(x)U^{4}_{\eps}\left(\frac{x}\eps\right)dx\lesssim
\intr \uvmu^2(x)\left(
U^4  \left(\frac{x-P_{\eps}}\eps\right)+U^4  \left(\frac{x+P_{\eps}}\eps\right)\right)dx\lesssim  \eps^N.
\end{equation}
As a direct consequence, we infer
$$
\|\Phi_{\eps}\|_{H^2(\R^N)}\lesssim\eps^\frac N 2 .
$$
This implies that $\Phi_{\eps}\in C^{0,\frac12}(\R^{N})$  
and \(\|\Phi_{\eps}\|_{C^{0,\frac12}(\R^{N})}\lesssim \eps^{\frac{N}2}\). 
\\
Now, we write \eqref{eq:V}
as 
$$-\Delta \Phi   + V(x) \Phi  =\underbrace{-3\mu_{1}\uvmu^{2}(x)\Phi(x)}_{=f_1(x)}+
\underbrace{\uvmu (x)U_{\eps}\left(\frac{x}\eps\right)}_{=f_2(x)}\ \hbox{in}\ \R^N
$$ 
and we observe that, reasoning as in \eqref{eq:ueps}, we deduce that
$$
\|f_1\|_{L^m(\R^N)}\lesssim\eps^\frac N2\ \hbox{and}\ \|f_2\|_{L^m(\R^N)}\lesssim\eps^\frac Nm,
\; \text{ for any $m\ge 2$};
$$
then, assumptions $\bf{(V_{2})}$ implies,  
$$ 
\|\Phi_{\eps}\|_{W^{2,m}(\R^N)}\lesssim\eps^\frac N m,
$$
so that, choosing  $m>N$,     Sobolev embedding
$W^{2,m}(\R^N)\hookrightarrow C^{1,1-\frac Nm}(\R^N)$ yields the claim.
\end{proof}
\begin{remark} \label{re:regphi}
It is useful to remark that  by Lemma \eqref{phi}  since $\nabla \Phi_{\eps}(0)=0$ we deduce
 \begin{equation}\label{key1}
| \Phi_{\eps}(y)-\Phi_{\eps}(0)|  \lesssim |y|^{2-\frac Nm}\|\Phi_{\eps}\|_{C^{1,1-\frac Nm}(\R^N)}\lesssim 
 \eps  ^{\frac Nm}|y|^{2-\frac Nm}
\end{equation}
Moreover, since $\nabla \uvmu (0)=0$ we also have
 \begin{equation}\label{key2}
| \uvmu(y)-\uvmu(0)|  \lesssim |y|^{2 }\|\uvmu\|_{C^{2}(\R^N)}\lesssim 
 |y|^{2 }.
\end{equation}
\end{remark}

\begin{lemma}\label{le:Psi}
There exists a unique  
$\Psi_\eps\in H^2 (\R^N)$,  solution of the equation 
\begin{equation}\label{eq:Psi}
-\Delta \Psi_\eps  + \left(\omega_0 -3\mu_{2} \left(U_{P_\eps}^{2}(x) +U_{-P_\eps}^{2}(x)\right) \right)\Psi_\eps = 2 \beta \Phi_\eps(0)\uvmu(0) U_{\eps}(x). \; 
\end{equation}
Moreover,   $\|\Psi_\eps\|_{H^2(\R^N)}\lesssim \eps^{N\over2}$ and 
 there exist $B,\, \gamma,\, R_{0} >0$ such that
\[
| \Psi_{\eps} (x)|\lesssim \eps^{N/2}\left( e^{-\gamma |x+\frac{P_{\eps}}\eps|}+e^{-\gamma |x-\frac{P_{\eps}}\eps|}\right),\qquad \forall\, |x|\geq R_{0}.
\]
\end{lemma}
\begin{proof}
As $U$ is radial, there exists a unique $\Psi $  radial solution to
\begin{equation}\label{p30} 
-\Delta \Psi   + \left(\omega_0 -3\mu_{2}U ^{2}(x)\right)\Psi =U(x), 
\end{equation}
then the   function $\Psi_\eps$  defined as
$$
\Psi_\eps(x)=2 \beta\Phi_\eps(0)\uvmu(0)\!\!\left[\Psi\left(x+\frac{P_\eps}\eps\right)+\Psi\left(x-\frac{P_\eps}\eps\right)\right]
$$ 
solves \eqref{eq:Psi}. The regularity properties of $\Psi$ are consequence of the regularity properties of $U$, while the bound from above of the $H^{2}(\R^{N})$ norm follows from the upper bound on the $L^{\infty}(\R^{N})$ norm of $\Phi_{\eps}$.

The exponential decay of $\Psi_{\eps}$ will follow from the analogous decay of $\Psi$.
In order to prove this property, let us first show that 
there exists $R>1$ such that $\Psi(r)\geq 0$ for every $r>R$.
\\
By contradiction there exists a sequence $r_{n}\to+\infty$ of minimum point at a 
negative level for $\Psi$. Then
\[\begin{split}
0>-\Psi_{rr}(r_{n})-\frac{N-1}r\Psi_{r}(r_{n})&=-\omega_{0}\Psi(r_{n})+3\mu_{2}U^{2}(r_{n})\Psi(r_{n})+U(r_{n})
\\
&\geq -\Psi(r_{n})\left(\omega_{0}-3\mu_{2}U^{2}(r_{n})
\right)\geq 0
\end{split}\]  
as soon as $r_{n}$ is sufficiently large so that the parenthesis is positive .
\\
Let us now fix $A>0$ such that $3\mu_{2}U\Psi+1\leq A$ and let $v(r)=Be^{-\gamma r}$ with $\gamma^{2}<\omega_{0}$.
Then $w=v-\Psi$  solves
\[\begin{split}
-\Delta w+\omega_{0}w&=v\left(\omega_{0}-\gamma^{2}+\frac{N-1}r\gamma\right)
-U(3\mu_{2}\Psi U+1)
\\
&\geq v\left(\omega_{0}-\gamma^{2}+\frac{N-1}r\gamma\right)-AU
\\
&\geq Be^{-\gamma r}\left(\omega_{0}-\gamma^{2}+\frac{N-1}r\gamma\right)-Ae^{-\sqrt{\omega_{0}}r}
\\
&= e^{-\gamma r}\left[B\left(\omega_{0}-\gamma^{2}+\frac{N-1}r\gamma\right)-
Ae^{-(\sqrt{\omega_{0}}-\gamma)r}
\right]>0
\end{split}\]
as $\gamma^{2}<\omega_{0}$ for $r>R_{1}$ sufficiently large. In addition,
we can choose $B$ such that $Be^{-\gamma r}\geq \max_{|x|=R_{1}}\Psi$, 
then the maximum principle yields $0\leq \Psi\leq v$ for $|x|>R_{1}$.
\end{proof}

\begin{remark}
Let us point out that assumptions $\bf{(V_{1})}$ and $\bf{ (V_{2})}$ are satisfied by constant 
potentials as well as by  potentials of the type $V(x)=|x|^{m}$ with $m>0$ as shown in 
\cite{bo, shen}.
\end{remark}

\begin{remark}
As a consequence of Sobolev embedding, the bounds from above stated in Lemma \ref{phi} hold for $\Psi_{\eps}$ as well.
\end{remark}

\subsection{Rewriting the problem}\label{formulazione}
Let us  introduce the orthogonal projections 
\[
\widetilde{\Pi}: L^{2}(\R^{N})\mapsto   \widetilde{K},\qquad \text{ and  }\widetilde{\Pi}^{\perp}: L^{2}(\R^{N}) \mapsto   \widetilde{K}^{\perp}
\]
where 
\[
\widetilde{K}:=\mathtt{span}\{(0,Z_\eps)\} \subset L^{2},\qquad
 \widetilde{K}^{\perp}:=\left\{(f,g)\in L^{2}(\R^{N})\times L^{2}(\R^{N}) : \intr gZ_{\eps}dx=0\right\}.
\]
Plugging the ansatz  $u=\Xi_\eps +\varphi$ and $v=\Theta_\eps+\psi$ (see \eqref{ans})  into \eqref{pro:P2}, one obtains  the following equivalent system
\begin{equation} \label{eq:sistpro}
\begin{cases} 
\widetilde{\Pi}\left\{\Lcal(\varphi,\psi)-\Ecal-\Ncal(\varphi,\psi) \right\}&=0 
\\
\widetilde{\Pi}^{\perp}\left\{\Lcal(\varphi,\psi)-\Ecal-\Ncal(\varphi,\psi) \right\}&=0 
\end{cases} 
\end{equation}
Here the linear operator $\Lcal=(\Lcal_{1},\Lcal_{2})$ is defined by
\begin{align}
&
\nonumber \Lcal_{1}(\varphi,\psi):=
-\Delta \varphi +V(x)\varphi-
\left(3\mu_{1} \Xi_\eps   ^{2}+\beta \Theta^{2}_{\eps}\left(\frac{x}\eps\right)\right)\varphi
- 2\beta \Theta_\eps  \left(\frac{x}\eps\right)\Xi_\eps \psi\left(\frac{x}\eps\right)
\\
\label{eq:L21}
&\Lcal_{2}(\varphi,\psi):=-\Delta\psi+W(\eps x)\psi 
-\left(3\mu_{2}\Theta_\eps  ^{2}+
\beta \Xi_\eps ^2 (\eps x) \right)\psi
-2\beta \Theta_\eps  \Xi_\eps  (\eps x) \varphi(\eps x).
\end{align}
The error term  $\Ecal=(\Ecal_{1}, \Ecal_{2}) $ is defined by
\begin{equation}\label{eq:E11}
\begin{split}
\Ecal_{1}:=&3\mu_1\Upsilon(x)\beta^2\Phi_\eps^2+\mu_1\beta^3\Phi_\eps^3+\beta^3 \Upsilon(x)\Psi_\eps\left(\frac x \eps\right)+2\beta^2\Upsilon(x)U_\eps\left(\frac x \eps\right)\Psi_\eps\left(\frac x \eps\right)
\\
&+\beta^2\Phi_\eps(x) U_\eps^2\left(\frac x \eps\right)
+\beta^4\Phi_\eps \Psi_\eps\left(\frac x \eps\right)+2\beta^2\Phi_\eps U_\eps\left(\frac x \eps\right)\Psi_\eps\left(\frac x \eps\right)
\end{split} \end{equation}
\begin{equation}\label{eq:E21}
\begin{split}
\Ecal_{2}:=& \left(\omega_0-\omega(\eps x)\right)
 \Theta_\eps  
 +2\beta^2 U_{\eps}\left(\Upsilon(\eps x)\Phi_\eps(\eps x)- \Phi_\eps(0)\Upsilon(0)\right)   
\\
& 
+\mu_{2}(U_\eps^3-U_{P_{\eps}}^3-U^{3}_{-P_{\eps}})  +3\beta\mu_{2}\left(U_{\eps}^{2}-U^{2}_{P_{\eps}}-U^{2}_{-P_{\eps}}\right)\Psi_{\eps}
\\
&+\beta^{3}\Phi^{2}_{\eps}(\eps x)\Theta_{\eps}+2\beta^{3}\Psi_{\eps}\Upsilon(\eps x)\Phi_{\eps}(\eps x)+3\mu_{2}\beta^{2} U_{\eps}\Psi_{\eps}^{2}
+\mu_{2}\beta^{3}\Psi_{\eps}^{3}
\end{split} \end{equation}
where  
\[
\omega (x):=W(x)-\beta \uvmu^2(x)  \quad \hbox{and}\quad \omega(0)=\omega_0.
\]
Finally,  the  nonlinear term $\Ncal=(\Ncal_{1},\Ncal_{2})$ is defined by
\begin{align*}
 \Ncal_{1}(\varphi,\psi):=& \mu_{1}\varphi^{2}\left(3\Xi_\eps +\varphi\right)
+\beta\varphi\psi\left(\frac{x}\eps\right)\left(2\Theta_\eps  \left(\frac{x}\eps\right)+\psi\left(\frac{x}
\eps\right)\right)+
\beta \Xi_\eps  \psi^{2}\left(\frac{x}\eps\right) 
\\ 
\Ncal_{2}(\varphi,\psi):=&  \mu_{2}\psi^{2}\left(3\Theta_\eps  +\psi\right)+\beta\psi\varphi(\eps x)
\left(2\Xi_\eps (\eps x)+\varphi(\eps x)\right) +\beta \Theta_\eps  \varphi^{2}(\eps x).
\end{align*}
 
\section{The linear theory}\label{sec:equainfdim}
Let us start the study of the second equation in \eqref{eq:sistpro} by proving the following crucial result.
\begin{lemma}\label{lem:Linv} 
There exists $c>0$,  and $\eps_{0}>0$ such that for every $\eps\in (0,\eps_{0})$
and  for every $\beta<0$ 
it results
\[
\|\widetilde{\Pi}^{\perp}\Lcal(\varphi,\psi)\|_{L^2(\R^N)\times L^2(\R^N)}\geq c\|(\varphi,\psi)\|_{H^2_V(\R^N)\times H^2_\eps (\R^N)}, \qquad \forall (\varphi,\psi) \in   K^{\perp}
\]
\end{lemma}

\begin{proof}

We argue by contradiction and suppose that there exist  $\eps_{n}\to 0$ and $(\varphi_{n},\psi_{n})\in  K^{\perp}$ with 
$\|(\varphi_{n},\psi_{n})\|_{H^2(\R^N)\times H^2(\R^N)}=1$ such that
\begin{equation}\label{eq:sistlemma1}
\left\{\begin{aligned}
-\Delta  \varphi_{n} +V(x) \varphi_{n} =&
\left[3\mu_{1}\Xi_{\eps_n}^{2}+\beta \Theta^{2}_{\eps_n}\left(\frac{x}{\eps_{n}}\right)\right]  \varphi_{n} +2\beta \Xi_{\eps_n} 
\Theta_{\eps_{n}}  \left(\frac{x}{\eps_{n}}\right)\psi_{n}\left(\frac{x}{\eps_{n}}\right)+f_n\\
-\Delta \psi_{n}+W(\eps_{n}x)\psi_{n}=&
\left[3\mu_{2}\Theta^{2}_{\eps_n}+ \beta \Xi_{\eps_n} ^{2}(\eps_{n} x)\right]
\psi_{n}+2\beta \Theta_{\eps_n}  \Xi_\eps{_{n}} (\eps_{n}x)\varphi_{n}(\eps_{n}x)+g_n
\\
& +t_{n}Z_{\eps _{n}}
\end{aligned}\right.
\end{equation}
where    $f_{n},\, g_n \to 0$ in $L^2(\R^N)$,
$\intr f_nZ_{\eps _{n}}=\intr g_nZ_{\eps _{n}}=0$,  $t_{n}\in \R$ and  $Z_{\eps _{n}}$ is given in \eqref{eq:defU2}.

{\bf Step 1.}
Let us first  show that $\varphi_{n}\to 0$ strongly  in $H^{2}_V(\R^{N})$. Then  by 
Sobolev embeddings $\varphi\to 0$ in $L^\infty(\R^N).$  
As $\varphi_{n}$ is bounded in $H^{1}$ , there exists $\varphi\in H^{1}(\R^{N})$
such that, up to  a subsequence, $\varphi_{n}\to \varphi$ weakly in $H^1({\R^N}), $ 
strongly in $L^2_{{\rm loc}}(\R^N)$ and almost everywhere in $\R^N.$ 
Taking into account that
$$
\Xi_{\eps_n}^2=\Upsilon^2+\beta\Phi_{\eps_n}(\beta\Phi_{\eps_n}+2\Upsilon)
$$
and testing the first equation in \eqref{eq:sistlemma1} by 
$\chi\in C^\infty_0(\R^N)$ we get
$$
 \begin{aligned}
\intr\nabla  \varphi_{n}\nabla \chi + V(x)  \varphi_n\chi
=&
\intr\left[3\mu_{1}  \Upsilon^{2}+3\mu_1\beta\Phi_{\eps_n}(\beta\Phi_{\eps_n}+2\Upsilon)+\beta \Theta_{\eps_{n}}^{2}\left(\frac{x}{\eps_{n}}\right)\right]  \varphi_{n}\chi
\\
& 
+\intr 2\beta (\Upsilon +\beta\Phi_{\eps_n} )
\Theta_\eps  \left(\frac{x}{\eps_{n}}\right)\psi_{n}\left(\frac{x}{\eps_{n}}\right)\chi+\intr f_n\chi .
\end{aligned} 
$$
By applying Lemma \ref{phi} one deduces that
\[
\left|\intr \Phi_{\eps}(\beta\Phi_{\eps_n}+2\Upsilon) \varphi_{n}\chi\right|\to 0.
\]
Moreover,  by applying Lemma \ref{le:Psi} we obtain
\[
\intr \Theta_{\eps_{n}}^{2}\left(\frac{x}{\eps_{n}}\right)\left|\varphi_{n}\chi\right|
\lesssim
\|\varphi_{n}\|_{2}
\int_{\R^{N}} \left[\Theta^{4}_{\eps_{n}}\left(\frac{x}{\eps_{n}}\right)\right]^{\frac12}
\lesssim \eps_{n}^{\frac{N}2}\|\Theta_{\eps_{n}}^{2}\|_{L^{2}(\R^{N})}.
\]
Arguing analogously on the other terms on the right hand side, it follows that 
$\varphi$ solves
\[
-\Delta \varphi+V(x)\varphi=3\mu_{1}\uvmu^2\varphi\ \hbox{in}\ \R^N.
\]
Since  $\varphi$ is an even function, by the assumtpion ${\bf(V_1)}$ we get $\varphi\equiv 0$.
\\
In order to show that $\varphi_{n}\to 0$ strongly in $H^2_{V}(\R^N)$, it is enough to exploit Lemma  \ref{phi} and  \ref{le:Psi}  to verify that the $L^2(\R^N)-$norm of  the right hand side (R.H.S. for short) of the first equation goes to zero, and 
then apply hypothesis $\bf{(V_{1})}$.
Indeed, as $\|(\varphi_{n},\psi_{n})\|_{H^2_V(\R^N)\times H^2_{W}(\R^N)}=1$, Sobolev embeddings yield
 $\|\varphi_n\|_{L^2(\R^N)},\|\psi_n\|_{L^2(\R^N)}\le1$ and $\|\varphi_n\|_{L^\infty(\R^N)}\le1$. Then 
$$
\begin{aligned}
\|R.H.S.\|_{L^2(\R^N)}
\lesssim 
&
\|\Upsilon^{2}\varphi_n\|_{L^2(\R^N)}
+\left\|\left[\Phi_{\eps_n}(\beta\Phi_{\eps_n}+2\Upsilon)+
\beta \Theta^{2}_{\eps_n}\left(\frac{x}{\eps_{n}}\right)\right]  \varphi_{n}\right\|_{L^2(\R^N)}
\\
& +\left\| (\Upsilon +\beta\Phi_{\eps_n} )
\Theta_\eps  \left(\frac{x}{\eps_{n}}\right)\psi_{n}\left(\frac{x}{\eps_{n}}\right)\right\|_{L^2(\R^N)}+\|f_n\|_{L^2(\R^N)}
\\
\lesssim &
  \|\Upsilon^{2}\varphi_n\|_{L^2(\R^N)}+\|\Phi_{\eps_n}\|_{L^\infty(\R^N)} \|\varphi_n\|_{L^2(\R^N)}+\eps^\frac N2 \|\varphi_n\|_{L^\infty(\R^N)}\\
&+\eps^\frac N2 \|\psi_n\|_{L^2(\R^N)}+\|f_n\|_{L^2(\R^N)}\\ =&o(1),
\end{aligned} 
$$
where we have also taken into account that $\Upsilon$ decays exponentially and $\varphi\to0$ strongly in $L^2_{{\rm loc}}(\R^N)$, so that 
we also have
$$
\intr\Upsilon^{4}\varphi_n^2=o(1) .
$$
{\bf Step 2.} We now  study the second equation in \eqref{eq:sistlemma1} and we prove that  
$t_{n}\to 0$.
We  test  with $ Z_{\eps _{n}}$  and we remind that $ Z_{\eps _{n}}$ solves
$$
 -\Delta Z_{\eps _{n}}   +\omega_0  Z_{\eps _{n}}  =3\mu_2\left(U^2_{P_{\eps _{n}}}\partial_{x_1}
 U_{P_{\eps _{n}}}- U^2_{-P_{\eps _{n}}}\partial_{x_1} U_{-P_{\eps _{n}}}\right).
 $$
Therefore we get

\begin{equation}\label{new2}\begin{aligned}
t_n\intr Z_{\eps _{n}}^2=&
  \intr \left( W(\eps_n x)-\beta \uvmu^{2}(\eps_{n} x)-\omega_0  \right)Z_{\eps_n}  \psi_n\\
  &-3\mu_2\intr 
\left[ U^{2}_{\eps_n}Z_{\eps_n}-\left(U^{2}_{-P_{\eps_n}}\partial_{x_{1}}U_{-P_{\eps_n}}-U^{2}_{P_{\eps_n}}
\partial_{x_{1}}U_{P_{\eps_n}}\right)\right]\psi_n
\\
&-2\beta\int_{\R^N} \Theta_{\eps_n}  \Xi(\eps_{n}x)\varphi_{n}(\eps_{n}x)Z_{\eps _{n}}
-2\beta^{2}\intr \uvmu(\eps_{n}x)\Phi_{\eps_{n}}(\eps_{n}x)\psi_{n}Z_{\eps _{n}}
\\
&
-\beta^{3}\intr\Phi^{2}(\eps_{n}x)\psi_{n}Z_{\eps _{n}}
-3\mu_{2}\beta\intr \Psi_{\eps_{n}}\left(2U_{\eps_{n}}+\beta\Psi_{\eps_{n}}\right)\psi_{n}Z_{n}
\\
\lesssim &
\left\|\left( W(\eps_n x)-\beta \uvmu^{2}(\eps_{n} x)-\omega _0 \right)Z_{\eps_n} \right\|_{L^2(\R^N)}\|\psi_n\|_{L^2(\R^N)}\\
&+\left\|U^{2}_{\eps_n}Z_{\eps_n}-\left(U^{2}_{-P_{\eps_n}}\partial_{x_{1}}U_{-P_{\eps_n}}-U^{2}_{P_{\eps_n}}
\partial_{x_{1}}U_{P_{\eps_n}}\right) \right\|_{L^2(\R^N)}\|\psi_n\|_{L^2(\R^N)}\\
&+ \|\varphi_n\|_{L^\infty(\R^N)}+ \|\Phi_{\eps_{n}}\|_{L^\infty(\R^N)}+\|\Psi_{\eps_{n}}\|_{L^\infty(\R^N)}
\\ 
=&o(1).
\end{aligned}
\end{equation}

Indeed, we use the exponential decay of $U$ and of  its derivatives. 
Since $\omega _0 =W(0)-\beta \uvmu^{2}(0)$ 
we  get
$$ 
|W(y)-\beta \uvmu^{2}(y)-\omega _0|\lesssim |y|\quad \hbox{if }\  |y|\le \sigma$$
for some $\sigma>0$ and so
$$
\left\|\left( W(\eps_n x)-\beta \uvmu^{2}(\eps_{n} x)-\omega _0 \right)Z_{\eps_n} \right\|_{L^2(\R^N)}=o(1).$$

Moreover a direct computation and Lemma \ref{ACR} shows that
$$\left\|U^{2}_{\eps_n}Z_{\eps_n}-\left(U^{2}_{-P_{\eps_n}}\partial_{x_{1}}U_{-P_{\eps_n}}-U^{2}_{P_{\eps_n}}
\partial_{x_{1}}U_{P_{\eps_n}}\right) \right\|_{L^2(\R^N)}=o(1).
$$ 
It is possible to show that
all the other integral terms on the left hand side tend to zero by applying Lemma
\ref{phi} and \ref{le:Psi}, Step 1. and Sobolev embeddings.

Finally, since it is immediate to check that $\| Z_{\eps _{n}}\|_{L^{2}(\R^{N})}=C+o(1)$ for some $C>0,$ we deduce that $t_n=o(1)$.\\
{\bf Step 3.} 
Let us now introduce the sequences
\[
\wt{\psi}_{-P_{n}}(x):=\psi_{n}\left(x-\frac{P_{\eps_{n}}}{\eps_{n}}\right), \qquad 
\wt{\psi}_{+P_{n}}(x):=\psi_{n}\left(x+\frac{P_{\eps_{n}}}{\eps_{n}}\right).
\]
We will show that (up to subsequences)  $\wt{\psi}_{\pm P_{n}}\rightharpoonup 0$
weakly in $H^1(\R^N)$ and strongly in $L^2_{{\rm loc}}(\R^N).$
\\
Both these sequences are bounded in $H^{2}_{W_{\eps_n}}(\R^{N})$, so that, up to subsequences,   $\wt{\psi}_{\pm P_{n}}\rightharpoonup \wt{\psi}_{\pm}$
weakly in $H^1(\R^N)$ and strongly in $L^2_{{\rm loc}}(\R^N).$
Let us first show that  $\wt{\psi}_{+}\equiv0$, then an analogous argument will yield that $\wt{\psi}_{+}\equiv 0$.\\
In the following we will use the notation $\wt{\psi}_{+P_{n}}(x)=\wt{\psi}_{P_{n}}(x)$.
Recalling \eqref{eq:sistlemma1}, the function $\wt{\psi}_{P_{n}}$ satisfies the equation
$$\begin{aligned}
-\Delta \wt{\psi}_{P_{n}}+W(\eps_{n}x+P_{n})\wt{\psi}_{P_{n}}&=
3\mu_{2}\left(U  (x)+U\left(x+2\frac{P_{n}}{\eps_{n}}\right)
+\beta\Psi_{\eps_{n}}  
\left(x+\frac{P_{n}}{\eps_{n}}\right)\right)^{2} \wt{\psi}_{P_{n}}
\\
& + \beta \uvmu^{2}(\eps_{n} x+P_{n}) \wt{\psi}_{P_{n}}
\\
&+\beta\Phi_{\eps_n}(\eps_{n} x+P_{n}) (\beta\Phi_{\eps_n}(\eps_{n} x+P_{n}) +2\Upsilon(\eps_{n} x+P_{n}) )\wt{\psi}_{P_{n}}\\
&+2\beta 
\Theta_{\eps_{n}}  \left(x+\frac{P_{n}}{\eps_{n}}
\right)\Xi_{\eps_{n}}  \left(x+\frac{P_{n}}{\eps_{n}}
\right)\varphi_{n}(\eps_{n}x+P_{n})
\\
& + g_{n}\left(x+\frac{P_{n}}{\eps_{n}}\right)+t_{n}Z_{\eps _{n}}\left(x+\frac{P_{n}}{\eps_{n}}\right).
\end{aligned}$$
Arguing as in the first step,  applying Lemma \ref{phi} and Lemma \ref{le:Psi},
and taking into account that $\|\varphi_{n}\|_{L^{\infty}(\R^{N})}\to 0$, 
we obtain that the limit function $\tilde{\psi}_+$ solves
the limit problem
\begin{equation}\label{eq:psi+}
-\Delta \wt{\psi}_{+}+\omega_{0}  \wt{\psi}_{+}=3\mu_{2}U^{2}\wt{\psi}_{+}\ \hbox{in}\ \R^N.
\end{equation}
On the other hand, the function $\wt{\psi}_{+}$ inherits the symmetry properties of the function $\wt{\psi}_{\pm P_{n}}$, namely it is even in the last two variables and it satisfies the orthogonality condition
$$
\intr  \wt{\psi}_{+} (y) \partial_{1}U (y)dy=0,
$$ 
as
\[
\begin{split}
0&=\intr \psi_n(x) Z_{\eps_n}(x)dx=\intr \wt{\psi}_{+ P_{n}}(y) 
\left(   \partial_{1}U (y)-\partial_{1}U \left(y-2\frac{P_{n}}{\eps_{n}}\right)\right)
dy
\\
&=\intr  \wt{\psi}_{+} (y) \partial_{1}U (y)dy+o(1).
\end{split}
\]
This, together with \eqref{eq:psi+}, yields $\wt{\psi}_{+}\equiv 0.$

{\bf Step 4.} Let us prove that a contradiction arises.
First, let us prove that $\left\|  \psi_{n}\right\|_{L^2(\R^N)}=o(1)$. By testing the second equation with $\psi_n$, and recalling that $\beta<0$ in view of \eqref{eq:ipobeta}, we deduce that
$$
\begin{aligned}
\intr |\nabla \psi_{n}|^2+W(\eps_{n}x)\psi_{n}^2&=
\intr 3\mu_{2}U^{2}_{\eps_n}\psi_n^2 +\beta \intr \Upsilon^2(\eps_n x)\psi_n^2 
\\
&+\intr \beta \Phi_{\eps_n} (\eps_{n} x) (\beta \Phi_{\eps_n} (\eps_{n} x)+2\Upsilon (\eps_{n} x))
\psi^2_{n}
\\ &+
\intr2\beta U_{\eps_n}  \Xi_\eps (\eps_{n}x)\varphi_{n}(\eps_{n}x)\psi_n +\intr g_n\psi_n 
\\
&
\lesssim \intr 3\mu_{2}U^{2}_{\eps_n}\psi_n^2+o(1),
\end{aligned}$$
where we have repeatedly applied Lemma \ref{phi}, \ref{le:Psi}, that 
$\|\phi_{n}\|_{L^{\infty}(\R^{N}})=o(1)$ and that $g_{n}\to 0$ strongly in $L^{2}(\R^{N})$.
Concerning the last term,  
we have that
$$\begin{aligned}
\intr U^{2}_{\eps_n}\psi_n^2&=\intr U^2 _{P_{\eps_n}}\psi_n^2+\intr U^2 _{-P_{\eps_n}}\psi_n^2+2\intr U  _{P_{\eps_n}}U  _{-P_{\eps_n}}\psi_n^2\\&\lesssim \intr  U^2  \wt{\psi}^2_{+P_{n}}+\intr  U^2  \wt{\psi}^2_{-P_{n}} +\|\psi_n\|_{L^\infty(\R^N)}\intr U  _{P_{\eps_n}}U  _{-P_{\eps_n}}=o(1)\end{aligned}$$
because $ \wt{\psi} _{\pm P_{n}}\to0 $ strongly in $L^2_{{\rm loc}}(\R^N)$ (as shown in the previous step) and $U$ decays exponentially.
This  implies that $\psi_{n}\to 0$ strongly in $H^{1}(\R^{N})$, thanks to \eqref{eq:W0}.
\\
Finally, let us prove that a contradiction arises by showing that also $\psi_{n}\to 0$ strongly in $H^2_{W_{\eps}}(\R^N).$
In order to show this, it is enough to use hypothesis ${\bf (V_{2})}$ and  to check that the $L^2(\R^N)-$norm of  
the right hand side  of the second equation in \eqref{eq:sistlemma1} goes to zero.
Indeed, by Lemma \ref{phi},  taking into account that $\|\psi_{n}\|_{L^2(\R^N)}\to 0$ and $\|\varphi_{n}\|_{L^\infty(\R^N)}\to0$  
$$\begin{aligned}
\|R.H.S.\|_{L^2(\R^N)}&\lesssim \|3\mu_{2}U^{2}_{\eps_n}\psi_n\|_{L^2(\R^N)}
+\left\|\left(\Upsilon(\eps_n x)+\beta\Phi_{\eps_n}(\eps_n x)\right)^2 \psi_{n}\right\|_{L^2(\R^N)}\\
& +\left\|  
U_{\eps_n}\left(\Upsilon(\eps_n x)+\beta\Phi_{\eps_n}(\eps_n x)\right)\varphi_{n}\left( {\eps_{n}x}\right)\right\|_{L^2(\R^N)}+\|g_n\|_{L^2(\R^N)}+|t_n|\|Z_{\eps _{n}}\|_{L^2(\R^N)}\\
&\lesssim  \|\psi_n\|_{L^2(\R^N)}+\|\varphi_n\|_{L^\infty(\R^N)}+\|g_n\|_{L^2(\R^N)}+|t_n|\\ &=o(1)
\end{aligned} $$
\end{proof}
 \begin{remark}
Let us observe that the hypothesis $\beta<0$ is needed only in the proof of Step 4. Moreover,
it is not needed in the case of a constant potential $W(x)\equiv W(0)=W_{0}$. Indeed, in this case if $\beta \geq 0$
\[
\beta  \Upsilon^{2}(x)\leq \beta \Upsilon^{2}(0)<W(0),\qquad \text{as $\omega_{0}>0$.}
\]
so that the sequence $\left(W_{0}-\beta \Upsilon^{2}(\eps_{n} x)\right)\psi_{n}^{2}\geq 0$ for every $x$ and the final contradiction can be obtained applying Fatou Lemma. However, 
even if at this point we can manage $\beta \geq 0$ (in the case of $W$ constant),
the study of the finite dimensional problem will require $\beta<0$ as shown in hypothesis
\eqref{eq:ipobeta}.
\end{remark}


\subsection{The size of the error term}\label{size}
In this subsection we compute the $L^{2}(\R^{N})$ of ${\mathcal E}$ which will determine 
the norm of the remainder term $(\varphi,\psi)$.
\begin{proposition}\label{pro:E}
There exists $\eps_{0}>0$ such that for every $\eps\in (0,\eps_{0})$ it results
$$
\|\Ecal_{1}\|_{L^2(\R^N)}=\Ocal(\eps^{N})
\quad \text{and}\quad
 \|\Ecal_{2}\|_{L^2(\R^N)}=\Ocal\left(\eps^{2}|\ln\eps|^{2}\right)
$$
so that, 
$$\|\Ecal\|_{L^2(\R^N)\times L^2(\R^N)}=
\Ocal(\eps^{2}|\ln\eps|^{2})$$
\end{proposition}
\begin{proof} 
Let us start studying $\Ecal_{1}$ given in \eqref{eq:E11}. We have
$$\begin{aligned}\|\Ecal_{1}\|_{L^2(\R^N)}&\lesssim \|\Upsilon\Phi_\eps ^{2}\|_{L^2(\R^N)}+\|\Phi_\eps ^{3}\|_{L^2(\R^N)}+\left\|
\Phi_\eps  U^{2}_{\eps}\left(\frac{\cdot}\eps\right)\right\|_{L^2(\R^N)}+\left\|\Upsilon\Psi_\eps\left(\frac{\cdot}\eps\right)\right\|_{L^2(\R^N)}\\
&+\left\|\Upsilon U_\eps \left(\frac{\cdot}\eps\right)\Psi_\eps\left(\frac{\cdot}\eps\right)\right\|_{L^2(\R^N)}+\left\|\Phi_\eps\Psi_\eps\left(\frac{\cdot}\eps\right)\right\|_{L^2(\R^N)}+\left\|\Phi_\eps U_\eps \left(\frac{\cdot}\eps\right)\Psi_\eps\left(\frac{\cdot}\eps\right)\right\|_{L^2(\R^N)}.
\end{aligned}$$
Lemma \ref{phi} and Sobolev embedding imply that (see Remark \ref{re:regphi})
 $\|\Phi_\eps\|_{\infty}\lesssim \eps^\frac N 2$.
Then we deduce
$$\|\Upsilon\Phi_\eps ^{2}\|_{L^2(\R^N)}\lesssim \|\Phi_\eps\|^2_{\infty}\lesssim \eps^N,\quad \|\Phi_\eps ^{3}\|_{L^2(\R^N)}\lesssim \eps^\frac{3N}{2}$$
and
$$
\left\|
\Phi_\eps  U^{2}_{\eps}\left(\frac{\cdot}\eps\right)\right\|_{L^2(\R^3)}=
\left(\intr
\Phi_\eps^2  U^{4}_{\eps}\left(\frac{x}\eps\right)\right)^\frac12\lesssim
\eps^{\frac N 2}\left(\intr U^{4}_{\eps}\left(\frac{x}\eps\right)\right)^\frac12 \lesssim\eps ^{N}.
$$
Moreover by applying Lemma \ref{le:Psi} we obtain
$$
\left\|\Upsilon\Psi_\eps\left(\frac{\cdot}\eps\right)\right\|_{L^2(\R^N)}\lesssim \eps^{\frac N 2}\|\Psi_\eps\|_{H^2(\R^N)}\lesssim \eps ^{N}.
$$
As far as concern the last three terms, similar computations show that 
$$\left\|\Upsilon U_\eps \left(\frac{\cdot}\eps\right)\Psi_\eps\left(\frac{\cdot}\eps\right)\right\|
_{L^2(\R^N)}+\left\|\Phi_\eps\Psi_\eps\left(\frac{\cdot}\eps\right)\right\|_{L^2(\R^N)}+\left\|
\Phi_\eps U_\eps \left(\frac{\cdot}\eps\right)\Psi_\eps\left(\frac{\cdot}\eps\right)\right\|
_{L^2(\R^N)}\lesssim \eps^{N}.
$$
Let us now  study the $L^{2}(\R^{N})$ norm of the terms in \eqref{eq:E21}. In view of 
\eqref{eq:W0}  and \eqref{eq:peps}, recalling that $x_0=0$ is a critical point of $\omega$ by symmetry, 
\begin{equation}\label{eq:taylor}
\begin{split} 
\|(\omega_0-\omega(\eps x)U_\eps\|_{L^2(\R^N)}&\lesssim \left(\intr \eps^4|x|
^4 U^2\left(x-\frac{P_\eps}{\eps}\right)\, dx\right)^{\frac 12}
\\
&\lesssim \left(\intr \eps^4\left|x+
\frac{P_\eps}{\eps}\right|^4 U^2\left(x\right)\, dx\right)^{\frac 12}
\\ &\lesssim \rho_\eps^2\lesssim \eps^2|\ln\eps|^2 .
\end{split}\end{equation}
On the other hand Lemma \ref{le:Psi} allows us to deduce that
$$ 
\|(\omega_0-\omega(\eps x)\Psi_\eps\|
_{L^2(\R^N)}\lesssim  
\intr \left(\eps^{4} \left|x+\frac{P_{\eps}}{\eps}\right|^{4} \Psi^{2} (x)\right)^{\frac12}
\lesssim \rho_{\eps}^{2}\|\Psi \|_{L^2(\R^N)}
\lesssim \rho_\eps^2\lesssim \eps^2|\ln\eps|^2 .
$$
By \eqref{key1} one has
\[
|\Phi_{\eps}(\eps x)-\Phi_{\eps}(0)|\leq \eps^{\frac{N}m}|\eps x|^{2-\frac{N}m}=\eps^{2}|x|^{1-\frac{N}m}
\]
so that, we obtain
$$\begin{aligned}
\|U_{\eps}\left(\Upsilon(\eps x)\Phi_\eps(\eps x)- \Phi_\eps(0)\Upsilon(0)\right) \|_{L^2(\R^N)}
\lesssim &
\|U_\eps \Upsilon(\eps x)(\Phi_\eps(\eps x)-\Phi_\eps(0))\|_{L^2(\R^N)}
\\
&+\|U_\eps (\Upsilon(\eps x)-Y(0))\Phi_\eps(0)\|_{L^2(\R^N)}
\\
\lesssim & \eps^2  \left(\intr \left(U(x)\left|x+\frac{P_\eps}\eps\right|^{ 2-\frac Nm}\right)^2\, dx\right)^{\frac 12}\\ 
&+\eps^2\|\Phi_\eps\|_{L^\infty(\R^N)}  \left(\intr \left(U(x)\left|x+\frac{P_\eps}\eps\right|^{ 2}\right)^2\, dx\right)^{\frac 12}\\ 
\lesssim &
  \eps^2\left(\frac{\rho_\eps}\eps\right)^{ 2-\frac Nm}+ \eps^{2+\frac N2}\left(\frac{\rho_\eps}\eps\right)^{ 2}\\ 
  \lesssim &\eps^2|\ln \eps|^{2-\frac N  m}+\eps^{2+\frac N2}|\ln\eps|^2
\lesssim \eps^2|\ln \eps|^{2-\frac N  m}
\end{aligned}
$$ 
By  Lemma \ref{ACR} - (i) one also has 
$$
\begin{aligned}
\|U_\eps^3-U_{P_{\eps}}^3-U^{3}_{-P_{\eps}}\|_{L^2(\R^N)}&\lesssim \|U_{P_\eps}^2U_{-P_\eps}\|_{L^2(\R^N)}+\|U_{-P_\eps}^2U_{P_\eps}\|_{L^2(\R^N)}
\\
&\lesssim \left(\intr U^2\left(x-2\frac{P_\eps}{\eps}\right)U^4(x)\, dx\right)^{\frac 12}
\lesssim e^{-2\sqrt{\omega_0}\frac{\rho_\eps}{\eps}}\left(\frac{\rho_\eps}{\eps}\right)^{-\frac{N-1}{2}}\\
&\lesssim \eps\rho_\eps\lesssim \eps^2|\ln\eps|.
 \end{aligned}
 $$
Moreover, conclusion (ii) of Lemma \ref{ACR} and \eqref{eq:peps} yield
$$
\begin{aligned}
\|\left(U_{\eps}^{2}-U^{2}_{P_{\eps}}-U^{2}_{-P_{\eps}}\right)\Psi_{\eps}\|_{L^2(\R^N)}
&\lesssim \|\Psi_\eps\|_\infty \left(\intr U^2\left(x-\frac{2P_\eps}{\eps}\right)U^2(x)\, dx\right)^{\frac 12}
\\
&
\lesssim \eps^{\frac{N}{2}}\left\{\begin{aligned}& e^{-2\sqrt{\omega_0}\frac{\rho_\eps}{\eps}}\left(\frac{\rho_\eps}{\eps}\right)^{-\frac 14}\ &\hbox{if}\ N=2
\\
& e^{-2\sqrt{\omega_0}\frac{\rho_\eps}{\eps}}\left(\frac{\rho_\eps}{\eps}\right)^{-1}\left|\ln\frac{\rho_\eps}{\eps}\right|^{\frac 12} \,  
&\hbox{if}\ N=3\end{aligned}\right.\\
&=o(\eps^2|\ln\eps|^2)
\end{aligned}
$$ 
as, by using \eqref{eq:peps} it follows that $\rho_\eps\sim \frac{1}{\sqrt{\omega_0}}\eps|\ln\eps|$ and hence
$\eps^{\frac N 2}e^{-2\sqrt{\omega_0}\frac{\rho_\eps}{\eps}}\left(\frac{\rho_\eps}{\eps}\right)^{-\frac{N-1}{2}}=o(\eps^2|\ln\eps|^2)$. 

Let us finally study the last terms in \eqref{eq:E21}
$$\|\Phi^{2}_{\eps}(\eps x)\Theta_{\eps}\|_{L^2(\R^N)}\lesssim \|\Phi_\eps\|^2_\infty\lesssim \eps^N; \quad \|\Psi_{\eps}\Upsilon(\eps x)\Phi_{\eps}(\eps x)\|_{L^2(\R^N)}\lesssim \eps^N$$
and
$$\|U_{\eps}\Psi_{\eps}^{2}\|_{L^2(\R^N)}\lesssim \eps^N;\quad \|\Psi_{\eps}^{3}\|_{L^2(\R^N)}
\lesssim \eps^{\frac{3N}{2}},$$
concluding the proof.
\end{proof}

\subsection{Solving  the second  equation in \eqref{eq:sistpro} }\label{fixed}
 Lemma \ref{lem:Linv} and  Proposition  \ref{pro:E} yields the following result
\begin{proposition}\label{pro:eqorto}

There exists $\eps>0$ such that for every $\eps\in (0,\eps_{0})$ there exists a unique
solution $(\varphi,\psi)\in K^{\perp}$ of the equation
\[
\widetilde{\Pi}^{\perp}\left(\Lcal(\varphi,\psi)-\Ecal-\Ncal(\varphi,\psi)\right)=0.
\]
Furthermore,
\begin{equation}\label{rate}
  \|(\varphi,\psi)\|_{X }=  \Ocal\left(\eps^2|\ln\eps|^{2}\right).
\end{equation}
\end{proposition}
\begin{proof}
We will obtain the result by applying the contraction principle to the continuous map
\[
T:C_{\eps}:=\left\{(\varphi,\psi)\in  K^{\perp}, : \|(\varphi,\psi)\|_{X}\leq A\tau_{\eps}\right\}\mapsto
C_{\eps}
\qquad T(\varphi,\psi):= {\bar\Lcal}\left[\widetilde{\Pi}^{\perp}\left(\Ecal+\Ncal(\varphi,\psi)\right)\right]
\]
where 
$$
\tau_{\eps}=\eps^{2}|\ln \eps|^{2},\qquad  \bar\Lcal:=\left(\widetilde{\Pi}^{\perp}\circ \Lcal\right)^{-1},
$$
is well defined thanks to Lemma \ref{lem:Linv},  and  $A$  is  a suitable positive constant  to  be chosen.
In order to find $A$, it is sufficient to prove that
\begin{equation}\label{eq:N}
\|\Ncal(\varphi,\psi)\|_{L^{2}(\R^{N})}=o(\tau_{\eps}),\qquad \text{forall $(\varphi,\psi)\in C_{\eps}$ }.
\end{equation}
Let us start studying $\Ncal_{1}$, taking into account that $\|\Xi_{\eps}\|_{L^{\infty}(\R^{N})}\leq C$,  it results
\begin{align}
\nonumber
\|\Ncal_{1}(\varphi,\psi)\|_{L^{2}(\R^{N})}^{2} 
\lesssim &
\intr  \varphi^{6}+\Xi_{\eps}^2\varphi^{4}+\psi^{2}\left(\frac{x}\eps\right)\varphi^{2}
\Theta^{2}_{\eps}\left(\frac{x}\eps\right)+\psi^{4}\left(\frac{x}\eps\right)\varphi^{2}+
 \Xi^2_{\eps}\psi^{4}\left(\frac{x}\eps\right)
\\
\nonumber\lesssim &
\|\varphi\|_{V}^{6}+\|\varphi\|_{V}^{4}+\| \varphi\|_{4}^{2}\left(\intr \psi^{4}\left(\frac{x}\eps\right)\Theta_\eps  ^{4}\left(\frac{x}\eps\right)\right)^{\frac12}
+ \left(\intr\psi^{6}\left(\frac{x}\eps\right)\right)^{\frac23}
\| \varphi\|_{6}^{2}
\\
\nonumber&+ \intr \psi^{4}\left(\frac{x}\eps\right)
\\
\nonumber\lesssim&
\|\varphi\|_{V}^{6}+\|\varphi\|_{V}^{4}+\eps^{\frac{N}2}\|\varphi\|_{V}^{2}\|\psi\|_{ \eps}^{2}+\eps^{\frac{2N}3}
\|\varphi\|_{V}^{2}\|\psi\|_{\eps}^{4}+\eps^{N}\|\psi\|_{\eps}^{4}
\\
\label{eq:N1piccolo}\lesssim&
\tau_{\eps}^{6}+\tau_{\eps}^{4}+\eps^{\frac{N}2}\tau_{\eps}^{4}+\eps^{\frac{2N}3}\tau^{6}_{\eps}+\eps^{N}\tau^{4}_{\eps}=o(\tau^{2}_{\eps}).
\end{align}
In addition, 
\begin{align*}
\|\Ncal_{2}(\varphi,\psi)\|^{2}_{L^{2}(\R^{N})}  \lesssim &
\intr \psi^{6}+\Theta_\eps  ^{2}\psi^{4}+\psi^{2}\varphi^{2}(\eps x)\Xi_{\eps}^{2}(\eps x)+
\psi^{2}\varphi^{4}(\eps x)+\Theta^{2}_{\eps}\varphi^{4}(\eps x)
\\
\lesssim &
\|\psi\|_{\eps}^{6}+\|\psi\|_{\eps}^{4}+\|\psi\|_{4}^{2}\left(
\intr \varphi^{4}(\eps x)\right)^{\frac12}+
\|\psi\|_{6}^{2}\left(\intr \varphi^{6}(\eps x)\right)^{\frac23}
\\
&+ \left(\intr \varphi^{6}(\eps x)\right)^{\frac23}
\\
\lesssim &
\|\psi\|_{ \eps}^{6}+\|\psi\|_{\eps}^{4}+\eps^{-\frac{N}2}\|\psi\|_{\eps}^{2}\|\varphi\|_{V}^{2}+\eps^{-\frac{2N}3}\|\psi\|_{\eps}^{2}\|\varphi\|_{V}^4+\eps^{-\frac{2N}3}\|\varphi\|_{V}^{4}
=o(\tau^{2}_{\eps} ).
\end{align*}
this, together with \eqref{eq:N1piccolo}, implies \eqref{eq:N}. Then, the claim follows by the contraction mapping theorem.
\end{proof}
\begin{remark}\label{cruciale}
Note that, by the Sobolev embeddings
$ H^2(\R^3)\hookrightarrow  C^{0,\frac12}(\R^3)$,  
$ H^2(\R^2)\hookrightarrow C^{0,\alpha}(\R^2)$ for any  $\alpha\in(0,1) $,
we deduce that
\begin{itemize}
\item $\|\varphi\|_ {C^{0,\frac12}(\R^3)}+\|\psi\|_ {C^{0,\frac12}(\R^3)}\lesssim \eps^2|\ln\eps|^{2}$,
\item $\|\varphi\|_ {C^{0,\alpha}(\R^2)}+\|\psi\|_ {C^{0,\alpha}(\R^2)}\lesssim \eps^2|\ln\eps|^{2}$, for every $\alpha\in (0,1)$,
\end{itemize}
so that 
\begin{itemize}
\item
$\dys |\varphi(\eps x)-\varphi(0)|\leq \eps^{2}|\ln\eps|^2(\eps|x|)^{1/2}=\eps^{\frac 52}|\ln\eps|^2|x|^{1/2}$ for $N=3$.
\item
$|\varphi(\eps x)-\varphi(0)|\leq \eps^{2}|\ln\eps|^2(\eps|x|)^{\alpha}=\eps^{2+\alpha}|\ln\eps|^2|x|^{\alpha}$ for $N=2$, for every $\alpha\in(0,1)$,
\end{itemize}
and analogous estimates hold for $\psi$.
\end{remark}

\section{Solving the reduced problem}\label{reduced}
In this section, we are going to study the first equation in \eqref{eq:sistpro}.
Let $(\varphi,\psi)$ the solution of the second equation in \eqref{eq:sistpro}, then
\[
\widetilde{\Pi}\left\{\Lcal(\varphi,\psi)-\Ecal-\Ncal(\varphi,\psi)\right\}=c_{0}Z_\eps 
\] 
where 
\begin{equation}\label{c0}
c_{0}:=\frac{\left(\Lcal_{2}(\varphi,\psi)-\Ecal_{2}-\Ncal_{2}(\varphi,\psi), Z_\eps \right)_{L^{2}(\R^{N})}}{\|Z_\eps \|
^{2}}.
\end{equation}

Our goal will be to prove that   $c_{0}=0$.
From now on, we fix  $(\varphi,\psi)$  given in Proposition \ref{pro:eqorto}.
\begin{lemma}\label{lem:L2ker}
It results 
\[
\left|\left(\Ncal_{2}(\varphi,\psi)-\Lcal_{2}(\varphi,\psi), Z_\eps \right)_{L^{2}(\R^{N})}\right|\leq o(\eps^{2}|\ln\eps|)
\]
\end{lemma}
\begin{proof} 
Arguing as in the proof of Proposition \ref{pro:eqorto} it is easy to obtain that
$$
\intr \Ncal_2(\varphi, \psi)Z_\eps\, dx =\Ocal (\|(\varphi, \psi)\|^2)=o(\eps^{2}|\ln \eps|).
$$
Now by using \eqref{eq:L21} and \eqref{eq:eqZ} we get that
\begin{align} 
\nonumber 
\intr \Lcal_2(\varphi, \psi)Z_\eps\, dx=&
\intr \nabla \psi\nabla Z_{\eps}+W(\eps x)\psi Z_{\eps}-
\left(3\mu_{2}\Theta_\eps  ^{2}+
\beta \Xi^2 (\eps x) \right)
\psi Z_\eps
\\
&\nonumber -2\beta\intr \Theta_\eps \Xi (\eps x) \varphi(\eps x) Z_\eps
\\
\nonumber=&
\intr\left(W(\eps x)-\omega_{0}\right)\psi Z_{\eps}
+
3\mu_2\left(U^2_{P_\eps}\partial_{x_1}U_{P_\eps}- U^2_{-P_\eps}\partial_{x_1} U_{-P_\eps}\right)
\psi
\\
\nonumber&-\intr\left(3\mu_{2}\Theta_\eps  ^{2}+
\beta \Xi^2 (\eps x) \right)
\psi Z_\eps-2\beta\intr \Theta_\eps \Xi (\eps x) \varphi(\eps x) Z_\eps
\\
\nonumber=&
\intr (\omega(\eps x)-\omega(0))Z_\eps\psi\, dx +3\mu_2\intr U_{P_\eps}^2\partial_1 U_{-P_\eps}\psi\, dx 
\\
&\nonumber-3\mu_2\intr U_{-P_\eps}^2\partial_1 U_{P_\eps}\psi
-6\mu_2\intr U_{P_\eps}U_{-P_\eps}Z_\eps\psi\, dx\\
\nonumber&-3\mu_2\beta^2\intr \Psi_\eps^2Z_\eps\psi\, dx -6\mu_2\beta\intr U_\eps\Psi_\eps Z_\eps\psi\, dx
 \\
\nonumber& -\beta^3\intr \Phi_\eps^2(\eps x)Z_\eps\psi\, dx -2\beta^2\intr \Upsilon(\eps x)\Phi_\eps(\eps x)Z_\eps\psi\, dx
 \\
\label{eq:psifizeta}&-2\beta^2 \intr \Psi_\eps \Upsilon(\eps x) \varphi(\eps x) Z_\eps\, dx-2\beta^3\intr \Psi_\eps \Phi_\eps(\eps x) \varphi(\eps x) Z_\eps\, dx
\\
\label{eq:ufizeta} &-2\beta\intr U_\eps \Upsilon(\eps x)\varphi(\eps x) Z_\eps\, dx-2\beta^2 \intr U_\eps \Phi_\eps(\eps x) \varphi(\eps x) Z_\eps\, dx.
\end{align} 
Let us study the right hand side. First of all, arguing as in \eqref{eq:taylor} and taking into account \eqref{rate}, we have
\[
\begin{split}
\intr (\omega(\eps x)-\omega(0))Z_\eps\psi\, dx 
&
\lesssim \left(\intr |\omega(\eps x)-\omega(0)|^2Z_\eps^2\, dx\right)^{\frac 12}\|\psi\|_{L^{2}(\R^{N})}
\\
&\lesssim \rho_\eps^2
\|\psi\|_{L^{2}(\R^{N})}=o(\eps^2|\ln\eps|).
\end{split}\]
\[
\begin{split}
\intr U_{P_\eps}^2\partial_1 U_{-P_\eps}\psi\, dx&\lesssim \left(\intr U_{P_\eps}^4 U_{-P_\eps}^2\, dx\right)^{\frac 12}\|\psi\|_{L^{2}(\R^{N})}
\\
&\lesssim e^{-2\sqrt{\omega_0}\frac{\rho_\eps}{\eps}}\left(\frac{\rho_\eps}{\eps}\right)^{-\frac{N-1}{2}}\|\psi\|_{L^{2}(\R^{N})}=o(\eps^2|\ln\eps|),
\end{split}\]
where we have applied Lemma \ref{ACR}.
Similarly 
$$
\intr U_{-P_\eps}^2\partial_1 U_{P_\eps}\psi=o(\eps^2|\ln\eps|),\qquad
\intr U_{P_\eps}U_{-P_\eps}Z_\eps\psi\, dx\lesssim \left(\intr U_{P_\eps}^4U_{-P_\eps}^2\, dx\right)^{\frac 12}\|\psi\|_{L^{2}(\R^{N})}=o(\eps^2|\ln\eps|).
$$ 
In addition, \eqref{rate} and Lemma \ref{phi},  \ref{le:Psi} yield
\[
\begin{split}
\intr \Psi_\eps^2Z_\eps\psi\, dx
&\lesssim \|\Psi_\eps\|_{L^{4}(\R^N)}^2\|\psi\|_{L^{2}(\R^{N})}\lesssim \eps^N\|\psi\|_{L^{2}(\R^{N})}=o(\eps^2|\ln\eps|).
\\
\intr U_\eps\Psi_\eps Z_\eps\psi\, dx
&\lesssim \|\Psi_\eps\|_{L^2(\R^N)}\|\psi\|_{L^{2}(\R^{N})}
\lesssim \eps^{\frac N 2}\|\psi\|_{L^{2}(\R^{N})}
\lesssim
\eps^{\frac N 2}\eps^2|\ln\eps|^2  =o(\eps^2|\ln\eps|)
\\
\intr \Phi_\eps^2(\eps x)Z_\eps\psi\, dx 
&\lesssim \|\Phi_\eps\|^2_{\infty}\|\psi\|_{L^{2}(\R^{N})}\lesssim \eps^N\|\psi\|_{L^{2}(\R^{N})}=o(\eps^2|\ln\eps|)
\\
\intr \Upsilon(\eps x)\Phi_\eps(\eps x)Z_\eps\psi\, dx
&
\lesssim \|\Phi_\eps\|_\infty\|\psi\|_{L^{2}(\R^{N})}\lesssim
\eps^{\frac N 2}\eps^2|\ln\eps|^2  =o(\eps^2|\ln\eps|)
\end{split} \]
and the terms in \eqref{eq:psifizeta} can be handled analogously.
Let us  focus  on the  terms in \eqref{eq:ufizeta}.
Recalling that 
\begin{equation}\label{eq:udu}
\intr U_{\pm P_\eps}\partial_1 U_{\pm P_\eps}\, dx=0, \qquad 
\end{equation}
we obtain
\begin{equation}\label{eq:last}
 \begin{split}
 \intr U_\eps \Upsilon(\eps x)\varphi(\eps x) Z_\eps\, dx=&
 \intr \Upsilon(\eps x)(\varphi(\eps x)-\varphi(0))U_\eps Z_\eps\, dx 
 \\
 &+\intr \varphi(0) (\Upsilon(\eps x)-\Upsilon(0))U_\eps Z_\eps\, dx +\Upsilon(0)\varphi(0)\intr U_\eps Z_\eps\, dx
 \\
=&\intr \Upsilon(\eps x)(\varphi(\eps x)-\varphi(0))U_\eps Z_\eps\, dx 
\\
&+\intr \varphi(0) (\Upsilon(\eps x)-\Upsilon(0))U_\eps Z_\eps\, dx  
 \\
 &-\Upsilon(0)\varphi(0)\intr \left(U_{P_\eps}\partial_1 U_{-P_\eps} -U_{-P_\eps}\partial_1 U_{P_\eps} \right)dx.
 \end{split}
\end{equation}
Taking into account Remark \ref{cruciale} (choosing $\alpha=\frac12$ for $N=2$), we infer
$$
\begin{aligned}
\intr \Upsilon(\eps x)(\varphi(\eps x)-\varphi(0))U_\eps Z_\eps\, dx&\lesssim 
\eps^{\frac 52}|\ln\eps|^2\intr |x|^{\frac 12}U_{P_\eps}^2\, dx\ 
\\
&\lesssim
\eps^{\frac 52}|\ln\eps|^2\intr \left|x+\frac{P_\eps}{\eps}\right|^{\frac 12}{U^2}\, dx\ 
\\
&\lesssim 
\eps^{2}\rho_\eps^{\frac 12}|\ln\eps|^2\ 
=o(\eps^2|\ln\eps|).
\end{aligned}
$$
Moreover, \eqref{eq:peps} and \eqref{rate} yield
$$
\intr \varphi(0) (\Upsilon(\eps x)-\Upsilon(0))U_\eps Z_\eps\, dx
\lesssim \eps^2\|\varphi\|_{\infty}
\intr |x|^2U_{P_\eps}^2\, dx \lesssim \rho_\eps^2\|\varphi\|_{V}=o(\eps^2|\ln\eps|).
$$
The last two terms in \eqref{eq:last} can be studied similarly,  by applying Lemma \ref{ACR}.
It results
$$
\Upsilon(0)\varphi(0)\intr U_{P_\eps}\partial_1 U_{-P_\eps}\, dx\lesssim \|\varphi\|_{\infty}\intr U_{P_\eps}U_{-P_\eps}\lesssim \|\varphi\|_{V} e^{-2\sqrt{\omega_0}\frac{\rho_\eps}{\eps}}\left(\frac{\rho_\eps}{\eps}\right)^{-\frac{N-1}{2}}=o(\eps^2|\ln\eps|).
$$ 
The last term in \eqref{eq:ufizeta} can be easier studied as
 \[
 \begin{split}
\left| \intr U_\eps \Phi_\eps(\eps x) \varphi(\eps x) Z_\eps\, dx\right|\lesssim \|\Phi_{\eps}\|_{\infty}\|\varphi\|_{\infty}\lesssim \eps^{N/2}\|\varphi\|_{V}=o(\eps^2|\ln\eps|)
 \end{split}
 \]
 concluding the proof.
\end{proof}
We are now in position  to study the relevant term in \eqref{c0}.
\begin{lemma}\label{lem:E2ker2} 
It results
\[
\intr\Ecal_{2}Z_\eps\, dx =\left[-\partial_{11}\omega(0)\mathfrak b\eps\rho_\eps-2\mu_2\mathfrak c e^{-2\sqrt{\omega_0}\frac{\rho_\eps}{\eps}}\left(\frac{\rho_\eps}{\eps}\right)^{-\frac{N-1}{2}}\right](1+o(1))\]
for some   positive constants $\mathfrak b$ and  $\mathfrak c$.
\end{lemma}
\begin{proof}
Recalling the definition of $\Ecal_{2}$ given in  \eqref{eq:E21} we obtain 
\begin{equation}\label{eq:E2Z}
\begin{aligned}
\intr\Ecal_{2}Z_\eps\, dx =&
\intr\left(\omega_0-\omega(\eps x)\right)\Theta_\eps Z_\eps\, dx  
+2\beta^2 \intr U_{\eps}\left(\Upsilon(\eps x)\Phi_\eps(\eps x)- \Phi_\eps(0)\Upsilon(0)\right) Z_\eps\, dx   
\\& 
+\mu_{2}\intr (U_\eps^3-U_{P_{\eps}}^3-U^{3}_{-P_{\eps}}) Z_\eps\, dx 
+3\beta\mu_{2}\intr\left(U_{\eps}^{2}-U^{2}_{P_{\eps}}-U^{2}_{-P_{\eps}}\right)\Psi_{\eps}Z_\eps\, dx 
\\
&+\beta^{3}\intr\Phi^{2}_{\eps}(\eps x)\Theta_{\eps}Z_\eps\, dx +2\beta^{3}\intr \Psi_{\eps}\Upsilon(\eps x)\Phi_{\eps}(\eps x)Z_\eps\, dx 
\\
&+3\mu_{2}\beta^{2}\intr U_{\eps}\Psi_{\eps}^{2}Z_\eps\, dx 
+\mu_{2}\beta^{3}\intr\Psi_{\eps}^{3}Z_\eps\, dx. \end{aligned}
\end{equation}
Direct computations yield
\begin{equation}\label{eq:1}
\begin{split}
\intr\left(\omega_0-\omega(\eps x)\right)\Theta_\eps Z_\eps\, dx  =&
\intr \left(\omega(0)-\omega(\eps x)\right)U_{\eps}(x) Z_\eps\, dx
\\&
 +\beta\intr \left(\omega(0)-\omega(\eps x)\right)\Psi_\eps Z_\eps\, dx\\
=&
\intr \left(\omega_0-\omega(\eps x)\right)U_{P_\eps}\partial_1 U_{P_\eps}\, dx
\\
&-
\intr \left(\omega_0-\omega(\eps x)\right)U_{P_\eps}\partial_1 U_{-P_\eps}\, dx\\
&+\intr \left(\omega_0-\omega(\eps x)\right)U_{-P_\eps}\partial_1 U_{P_\eps}\, dx
\\
&-
\intr \left(\omega_0-\omega(\eps x)\right)U_{-P_\eps}\partial_1 U_{-P_\eps}\, dx
\\
&+\beta\intr \left(\omega(0)-\omega(\eps x)\right)\Psi_\eps Z_\eps\, dx.
\end{split}
\end{equation}
Let us start studying the first term on the right hand side. Taking into account 
\eqref{eq:udu},  we obtain
\begin{align*}
\intr \left(\omega_0-\omega(\eps x)\right)U_{P_\eps}\partial_1 U_{P_\eps}\, dx
=&-\frac12 \intr \langle D^2\omega(0)\eps x, \eps x\rangle U\left(x+\frac{P_\eps}{\eps}\right)\partial_1U\left(x+\frac{P_\eps}{\eps}\right)\, dx
\\
&+{\mathcal O}\left(\intr \eps^3\left|x+\frac{P_\eps}{\eps}\right|^3U^2(y)\, dx\right)
\\
=&-\frac12 \intr \langle D^2\omega(0)(\eps y-P_\eps), \eps y-P_\eps\rangle U\left(y\right)\partial_1U\left(y\right)\, dy+o(\eps\rho_\eps)
\\
=&\frac 12\intr \langle D^2 \omega(0) \eps y, P_\eps\rangle U\left(y\right)\partial_1U\left(y\right)\, dy\\
&+\frac 12 \intr \langle D^2 \omega(0) P_\eps, \eps y\rangle U\left(y\right)\partial_1U\left(y\right)\, dy
+o(\eps\rho_\eps)
\\
=&\eps\rho_\eps\underbrace{\intr \langle D^2 \omega(0) P_0,  y\rangle U\left(y\right)\partial_1U\left(y\right)\, dy}_{:=A}+o(\eps\rho_\eps)\\
=&  \partial^2_{11} \omega(0)\mathfrak b\eps \rho_\eps +o(\eps \rho_\eps).
\end{align*}
Note that, from \eqref{eq:peps} and  from the fact that  $U$ is radially decreasing we deduce that
$$
A:=\partial^2_{11} \omega(0)\underbrace{\intr y_1^2U(y)\frac{U'(y)}{|y|}\, dy}_{<0}+\sum_{i=2}^N\partial^2_{1i}\omega(0)\underbrace{\intr y_1y_i U(y)\frac{U'(y)}{|y|}\, dy}_{=0}
$$ 
so that  
\begin{equation}\label{AAA}
A:=-\partial^2_{11} \omega(0)\mathfrak b\ \hbox{and}\ \mathfrak b:=-\intr y_1^2U(y)\frac{U'(y)}{|y|}\, dy>0
\end{equation} 
since $\partial^2_{11} \omega(0)<0$, thanks to \eqref{min-pot}.
In a similar way we get that 
$$
-\intr \left(\omega_0-\omega(\eps x)\right)U_{-P_\eps}\partial_1 U_{-P_\eps}\, dx=A\eps\rho_\eps+o(\eps\rho_\eps).
$$

We claim that the other terms on the right hand side of \eqref{eq:1} are of higher order with respect to $\eps\rho_{\eps}$. Indeed,  Lemma \ref{decay} and \eqref{eq:peps} yield
$$
\begin{aligned}
\intr \left(\omega_0-\omega(\eps x)\right)U_{-P_\eps}\partial_1 U_{P_\eps}\, dx&=-\frac 12 \intr \langle D^2\omega(0)\eps x, \eps x\rangle U\left(x-\frac{P_\eps}{\eps}\right)\partial_1 U\left(x+\frac{P_\eps}{\eps}\right)\, dx
\\
&+{\mathcal O}\left(\intr \eps^3|x|^3 U\left(x-\frac{P_\eps}{\eps}\right)U\left(x+\frac{P_\eps}{\eps}\right)\, dx\right)
\\
&=-\frac 12 \intr \langle D^2 \omega(0)P_\eps, P_\eps\rangle U\left(y-2\frac{P_\eps}{\eps}\right)\partial_1 U(y)\, dy+o(\eps\rho_\eps)\\
&=-\frac 12 \rho_\eps^2 \langle D^2 \omega(0)P_0, P_0\rangle \intr U\left(y-2\frac{P_\eps}{\eps}\right)\partial_1 U(y)\, dy+o(\eps\rho_\eps)\\
&=\frac 12\mathfrak c \rho_\eps^2\langle D^2 \omega(0)P_0, P_0\rangle e^{-2\sqrt{\omega_0}\frac{\rho_\eps}{\eps}}\left(\frac{\rho_\eps}{\eps}\right)^{-(N-1)+\frac{N+1}{2}}+o(\eps\rho_\eps)\\
&=\frac 12\mathfrak c \rho_\eps^2\langle D^2 \omega(0)P_0, P_0\rangle e^{-2\sqrt{\omega_0}\frac{\rho_\eps}{\eps}}\left(\frac{\rho_\eps}{\eps}\right)^{-\frac{N-1}{2}}\frac{\rho_\eps}{\eps}+o(\eps\rho_\eps)\\
&=o(\eps\rho_\eps).
\end{aligned}
$$

Similarly $$\intr \left(\omega_0-\omega(\eps x)\right)U_{P_\eps}\partial_1 U_{-P_\eps}\, dx=o(\eps\rho_\eps).$$
The last term in \eqref{eq:1} can be managed, by applying Lemma \ref{le:Psi}. Indeed, we have
$$\left|\intr \left(\omega(0)-\omega(\eps x)\right)\Psi_\eps Z_\eps\, dx\right|\leq c\eps^2\|\Psi_\eps\|_{L^2(\R^N)}\left(\intr |x|^4U^2_{P_\eps}\, dx\right)^{\frac 12}\leq \eps^{\frac N 2}\rho_\eps^2=o(\eps\rho_\eps).
$$
Let us now study the second term on the right hand side of  \eqref{eq:E2Z}. It holds
$$\begin{aligned}
\intr U_{\eps}\left(\Upsilon(\eps x)\Phi_\eps(\eps x)- \Phi_\eps(0)\Upsilon(0)\right) Z_\eps\, dx  &=\intr U_{\eps}\Upsilon(\eps x)\left(\Phi_\eps(\eps x)- \Phi_\eps(0)\right) Z_\eps\, dx\\
&+\intr U_{\eps}\left(\Upsilon(\eps x)- \Upsilon_\eps(0)\right)\Phi_\eps(0) Z_\eps\, dx.
\end{aligned}$$
In view of Remark \ref{re:regphi}, we obtain
$$
\begin{aligned}
\left|\intr U_{\eps}\Upsilon(\eps x)\left(\Phi_\eps(\eps x)- \Phi_\eps(0)\right) Z_\eps\, dx\right|&\lesssim  
\intr |\Phi_\eps(\eps x)-\Phi_\eps(0)|U_{P_\eps}^2\, dx\\ & \lesssim  \eps^2 \intr |x|^{2-\frac Nm}U_{P_\eps}^2\, dx
\\&\lesssim   \eps^{ 2}\left(\rho_\eps\over\eps\right)^{2-\frac Nm} =o(\eps\rho_\eps)\ \hbox{if}\ m>N
\end{aligned}
$$
and 
$$
\begin{aligned}
\left|\intr U_{\eps}\left(\Upsilon(\eps x)- \Upsilon_\eps(0)\right)\Phi_\eps(0) Z_\eps\, dx\right|\leq C\eps^2 \|\Phi_\eps\|_\infty\intr |x|^2 U_{P_\eps}^2\, dx\leq \eps^{\frac N 2}\rho_\eps^2=o(\eps\rho_\eps).
\end{aligned}
$$
Let us study the cubic and square term in $U_{P_{\eps}}$ in \eqref{eq:E2Z}. 
We have
$$
\begin{aligned}
\mu_{2}\intr (U_\eps^3-U_{P_{\eps}}^3-U^{3}_{-P_{\eps}}) Z_\eps\, dx  =&
3\mu_2\intr U_{P_\eps}^2U_{-P_\eps}Z_\eps\, dx+3\mu_2\intr U_{P_\eps}U_{-P_\eps}^2 Z_\eps\, dx\\
=&3\mu_2\intr U_{P_\eps}^2U_{-P_\eps}\partial_1 U_{P_\eps}\, dx-3\mu_2\intr U_{P_\eps}^2U_{-P_\eps}\partial_1 U_{-P_\eps}\, dx\\
&+3\mu_2\intr U_{P_\eps}U_{-P_\eps}^2\partial_1 U_{P_\eps}\, dx-3\mu_2\intr U_{P_\eps}U_{-P_\eps}^2\partial_1 U_{-P_\eps}\, dx\\
=&3\mu_2\intr U\left(y-2\frac{P_\eps}{\eps}\right)U^2(y)\partial_1 U(y)\, dy\\
&-3\mu_2\intr U^2\left(y+2\frac{P_\eps}{\eps}\right)U(y)\partial_1 U(y)\, dy\\
&+3\mu_2\intr U^2\left(y-2\frac{P_\eps}{\eps}\right)U(y)\partial_1 U(y)\, dy\\
&-3\mu_2\intr U\left(y+2\frac{P_\eps}{\eps}\right)U^2(y)\partial_1 U(y)\, dy.
\end{aligned}$$
By exploiting  Lemma \ref{decay}{\em -(i)} with $s=1$ and $t=3$ we deduce 
$$\begin{aligned}3\mu_2\intr U\left(y-2\frac{P_\eps}{\eps}\right)U^2(y)\partial_1 U(y)\, dy&=\mu_2\intr U\left(y-2\frac{P_\eps}{\eps}\right)\partial_1 U^3 (y)\, dy\\
&=-\mu_2\mathfrak c e^{-2\sqrt{\omega_0}\frac{\rho_\eps}{\eps}}\left(\frac{\rho_\eps}{\eps}\right)^{-\frac{N-1}{2}}(1+o(1))\end{aligned}
$$
and similarly
$$-3\mu_2\intr U\left(y+2\frac{P_\eps}{\eps}\right)U^2(y)\partial_1 U(y)\, dy=-\mu_2\mathfrak c e^{-2\sqrt{\omega_0}\frac{\rho_\eps}{\eps}}\left(\frac{\rho_\eps}{\eps}\right)^{-\frac{N-1}{2}}(1+o(1)).$$
Applying again Lemma \ref{decay}{\em -(ii)} with $s=t=2$ we infer
$$\begin{aligned}3\mu_2\intr U^2\left(y+2\frac{P_\eps}{\eps}\right)U(y)\partial_1 U(y)\, dy&=\frac 3 2\mu_2\intr U^2\left(y+2\frac{P_\eps}{\eps}\right)\partial_1 U^2(y)\, dy\\ &=\left\{\begin{aligned}&\bar{\mathfrak c}e^{-4\sqrt{\omega_0}\frac{\rho_\eps}{\eps}}\left(\frac{\rho_\eps}{\eps}\right)^{-\frac 12}(1+o(1))\quad &\hbox{if}\,\, N=2\\
&\bar{\mathfrak c}e^{-4\sqrt{\omega_0}\frac{\rho_\eps}{\eps}}\left(\frac{\rho_\eps}{\eps}\right)^{-2}\ln\frac{\rho_\eps}{\eps}(1+o(1))\quad &\hbox{if}\,\, N=3\\
\end{aligned}\right.\\
&=o\left(e^{-2\sqrt{\omega_0}\frac{\rho_\eps}{\eps}}\left(\frac{\rho_\eps}{\eps}\right)^{-\frac{N-1}{2}}\right)\end{aligned}$$
and similarly
$$
3\mu_2\intr U^2\left(y-2\frac{P_\eps}{\eps}\right)U(y)\partial_1 U(y)\, dy=o\left(e^{-2\sqrt{\omega_0}\frac{\rho_\eps}{\eps}}\left(\frac{\rho_\eps}{\eps}\right)^{-\frac{N-1}{2}}\right).
$$
Hence
$$\mu_{2}\intr (U_\eps^3-U_{P_{\eps}}^3-U^{3}_{-P_{\eps}}) Z_\eps\, dx=-2\mu_2\mathfrak c e^{-2\sqrt{\omega_0}\frac{\rho_\eps}{\eps}}\left(\frac{\rho_\eps}{\eps}\right)^{-\frac{N-1}{2}}(1+o(1)).$$
Again, by using Lemma \ref{le:Psi} and Lemma \ref{ACR},
it follows
$$
\begin{aligned}
\left|\intr\left(U_{\eps}^{2}-U^{2}_{P_{\eps}}-U^{2}_{-P_{\eps}}\right)\Psi_{\eps}Z_\eps\, dx \right|&\leq C \left[\intr U^2_{P_\eps}U_{-P_\eps}|\Psi_\eps|+\intr U^2_{-P_\eps}U_{P_\eps}|\Psi_\eps|\right]\\
&
\leq C \|\Psi_\eps\|_{L^2(\R^N)}\left(\intr U^2\left(y-2\frac{P_\eps}{\eps}\right)U^4(y)\, dy\right)^{\frac 12}\\
&
\leq C \eps^{\frac N 2}
e^{-2\sqrt{\omega_0}\frac{\rho_\eps}{\eps}}\left(\frac{\rho_\eps}{\eps}\right)^{-\frac{N-1}{2}}
\\
&=o\left(e^{-2\sqrt{\omega_0}\frac{\rho_\eps}{\eps}}\left(\frac{\rho_\eps}{\eps}\right)^{-\frac{N-1}{2}}\right).
\end{aligned}$$
By Lemma \ref{phi} and Lemma \ref{le:Psi} we  also obtain 
\[
\begin{split}
\left|\intr\Phi^{2}_{\eps}(\eps x)\Theta_{\eps}Z_\eps\, dx\right|
&
\leq C \intr |\Phi_\eps(\eps x)|^2 U_{P_\eps}^2\, dx+ C\intr |\Phi_\eps(\eps x)|^2|\Psi_\eps|\,dx \lesssim   \eps^N=o(\eps\rho_\eps)
\\
\left|\intr \Psi_{\eps}\Upsilon(\eps x)\Phi_{\eps}(\eps x)Z_\eps\, dx\right| &
\leq C \|\Phi_\eps\|_{\infty}\|\Psi_\eps\|_{L^2(\R^N)}\lesssim \eps^N=o(\eps\rho_\eps),
\end{split}\]
and the same upper bound holds for the last two terms in \eqref{eq:E2Z}, 
concluding the proof.
\end{proof}

\begin{proof}[Proof of Theorem \ref{thm:principal}: completed]

First notice that $(u,v)$ is a solution of \eqref{pro:P0} iff $(u,v)$ is a solution of \eqref{pro:P2}. Then, we look for a solution of \eqref{pro:P2} of the form given in \eqref{ans}. So that, we are lead to look for $(\varphi,\psi)\in X$ satisfying 
the system \eqref{eq:sistpro}.
Proposition \ref{pro:eqorto} yields the existence of $(\varphi, \psi)$ satisfying the second equation in \eqref{eq:sistpro}.  The first equation is solved as well, if we show that
$c_{0}=0$ where $c_{0}$ is given in \eqref{c0}.
Then, Lemma \ref{lem:L2ker} and \ref{lem:E2ker2} imply that we have  to find $d=d_\eps$ in \eqref{punti} such that  
$$
  \left[-\partial_{11}\omega(0)\mathfrak b\eps\rho_\eps-2\mu_2\mathfrak c e^{-2\sqrt{\omega_0}\frac{\rho_\eps}{\eps}}\left(\frac{\rho_\eps}{\eps}\right)^{-\frac{N-1}{2}}\right](1+o(1))=0,
$$
which is equivalent (taking into account that $\rho_\eps=d_\eps \eps\ln\frac1\eps$) to
$$(1-\sqrt\omega_0 d)\ln\eps+o(\ln\eps)=0$$
and this is solvable by $d_\eps$ such that $d_\eps\to \frac1{\sqrt{\omega_0}}$ as $\eps\to0$.
That concludes the proof.
 \end{proof}
\vskip6pt
\noindent
{\bf\Large Declarations. }
\vskip0.6cm
\noindent {\bf \large Conflict of interests. } On behalf of all authors, the corresponding author states that there is no conflict of interest.
\vskip5pt
\noindent {\bf \large Data Availability Statement.} Data sharing not applicable to this article as no datasets were generated or analyzed during the current study.

\appendix \section{Technical Lemma }\label{app:lemma}
\begin{lemma}[Lemma 3.7 in \cite{acr}]\label{ACR}
Let $u, v : \mathbb{R}^N \to \mathbb{R}$ be two positive continuous radial functions such that
\[
 u(x) \sim |x|^a e^{-b|x|}, \quad v(x) \sim |x|^{a'} e^{-b'|x|} 
 \]
as $|x| \to \infty$, where $a, a' \in \mathbb{R}$, and $b, b' >0$. Let $\xi \in \mathbb{R}^N $ such that $|\xi| \to \infty$. We denote $u_{\xi}(x)=u(x-\xi)$. Then the following asymptotic estimates hold:
\begin{itemize}
\item[(i)] If $b < b'$,
\[ 
\int_{\mathbb{R}^N} u_{\xi} v\sim e^{-b|\xi|} |\xi|^{a}. 
\]
A similar expression holds if $b > b'$, by replacing $a$ and $b$ with $a'$ and $b'$. 
\item[(ii)] If $b=b'$, suppose that $a \ge a'$. Then:
\[ \int_{\mathbb{R}^N} u_{\xi} v\sim
\begin{cases}
e^{-b|\xi|} |\xi|^{a+a'+\frac{N+1}{2}} & \text{ if } a' > -\frac{N+1}{2},\\
e^{-b|\xi|} |\xi|^{a} \log |\xi| & \text{ if } a' = -\frac{N+1}{2},\\
e^{-b|\xi|} |\xi|^{a}& \text{ if } a' < -\frac{N+1}{2}.
\end{cases} \]
\end{itemize}
\end{lemma}

\begin{lemma}[(Lemma A.2 in \cite{pv}]\label{decay}
Let $U_{\lambda, \mu}=\sqrt{\frac{\lambda}{\mu_1}}U(\sqrt{\lambda}x)$ where $U$ is the 
solution of \eqref{Uc}.
Let $s, t\geq 1$ and consider the fol\-lo\-wing integral
$$\Theta_{s, t}(\zeta):=\int_{\mathbb R^N}U^s_{\lambda, \mu}(x+\zeta)\partial_{x_1}U^t_{\lambda, \mu}(x)\, dx,\quad\zeta\in\mathbb R^N.$$
\begin{itemize}
\item[(i)] 
If $s<t$ then $$\Theta_{s, t}(\zeta)\sim \mathfrak c s\frac{\zeta_1}{|\zeta|}e^{-s\sqrt{\lambda}|\zeta|}|\zeta|^{-s\frac{N-1}{2}}\quad \hbox{as}\,\, |\zeta|\to+\infty;$$
\item[(ii)]
 If $s=t$ then $$\Theta_{s, t}(\zeta)\sim\left\{\begin{aligned}& \mathfrak c s\frac{\zeta_1}{|\zeta|}e^{-s\sqrt{\lambda}|\zeta|}|\zeta|^{-s(N-1)+\frac{N+1}{2}}\quad &\hbox{if}\,\, s<\frac{N+1}{N-1};\\
&\mathfrak c s\frac{\zeta_1}{|\zeta|}e^{-s\sqrt{\lambda}|\zeta|}|\zeta|^{-s\frac{(N-1)}{2}}\ln|\zeta|\quad &\hbox{if}\,\, s=\frac{N+1}{N-1};\\
&\mathfrak c s\frac{\zeta_1}{|\zeta|}e^{-s\sqrt{\lambda}|\zeta|}|\zeta|^{-s\frac{(N-1)}{2}}\quad &\hbox{if}\,\, s>\frac{N+1}{N-1}.\end{aligned}\right. $$ as $|\zeta|\to+\infty$.
\end{itemize}
Here $\mathfrak c$ denotes a positive constant. In case $N=1$ the first option in (ii) holds, namely
$$\Theta_{s, s}(\zeta)\sim\mathfrak c s\frac{\zeta_1}{|\zeta|}e^{-s|\zeta|}|\zeta|\quad \hbox{as}\,\,\ |\zeta|\to+\infty.$$
\end{lemma}

\bibliography{PartiallyConcentrating}
\bibliographystyle{abbrv}
\end{document}